\documentclass[11pt,oneside]{article}

\usepackage[round,colon,authoryear]{natbib}

\usepackage{graphicx}
\usepackage{amsmath}
\usepackage{amsfonts}
\usepackage{amssymb}
\usepackage{amsthm}
\usepackage{enumerate}
\usepackage{pdfsync}
\usepackage{comment}
\usepackage{color}
\usepackage{bm}
\usepackage{times}

     \usepackage[final,notref,notcite,color]{showkeys}

\usepackage{bm}

\bmdefine\taub{\tau}
\bmdefine\mub{\mu}
\bmdefine\lab{\lambda}
\bmdefine\varsigmab{\varsigma}

 \numberwithin{equation}{section}
    
\newcommand{\R}{\mathbb{R}}
\newcommand{\E}{\mathbb{E}}
\newcommand{\N}{\mathbb{N}}
\newcommand{\F}{\mathfrak{F}}
\newcommand{\1}{\mathbf{1}}
\newcommand{\Prob}{\mathbb{P}}
\newcommand{\Q}{\mathbb{Q}}
\newcommand{\ra}{\rightarrow}

\newcommand{\dx}{\mathrm{d}}
\newcommand{\la}{\lambda}

\bmdefine\thetab{\vartheta}

\theoremstyle{plain}
\newtheorem{thm}{Theorem}[section]
\newtheorem{lem}[thm]{Lemma}
\newtheorem{prop}[thm]{Proposition}
\newtheorem{cor}[thm]{Corollary}

\theoremstyle{definition}
\newtheorem{exmp}[thm]{Example}

\theoremstyle{remark}
\newtheorem{standassume}[thm]{Standing Assumption}
\newtheorem{assume}[thm]{Assumption}
\newtheorem{rem}[thm]{Remark}
\newtheorem{open}[thm]{Open Question}

\title{
DISTRIBUTION OF THE TIME TO EXPLOSION  FOR ONE-DIMENSIONAL DIFFUSIONS   
\thanks{We thank  Bahman Angoshtari, Peter Carr,  Zhenyu Cui, David Elworthy, Robert Fernholz,  Tomoyuki Ichiba, Kostas Kardaras, Monique Jeanblanc,  Alex Mijatovi\'c, Soumik Pal, Camelia Pop,    Paavo Salminen, Mykhaylo Shkolnikov, Kawei Wang, and    Phillip Whitman  for their insightful remarks that improved the manuscript.  We thank Alex Mijatovi\'c for raising the issues of strict positivity and full support  and for prompting us to think about them (Subsections~\ref{SP} and \ref{SS: FSP}, respectively).  We are grateful to the  referees for their meticulous reading of the paper and their many and incisive suggestions.  J.R.~acknowledges generous support from the Oxford-Man Institute of Quantitative Finance, University of Oxford, where a major part of this work was completed.
\newline $~~~~$ The   problem discussed here was suggested to us by Marc Yor, who also gave us generous  expert advice on several of its aspects. We dedicate the paper to his memory with deep sorrow at his passing. } 
 }

\vspace{2mm}

\author{  
\textsc{IOANNIS KARATZAS} \thanks{
Department of Mathematics,  Columbia University, New York, NY 10027 (E-mail: {\it ik@math.columbia.edu}), and       \textsc{Intech} Investment Management,  One Palmer Square, Suite 441, Princeton, NJ 08542    (E-mail:    {\it ik@enhanced.com}). Research   supported in part by  the National Science Foundation under  grant NSF-DMS-14-05210.}  
 \and
\textsc{JOHANNES RUF}                \thanks{ 
Department of Mathematics, University College London, Gower Street, London WC1E 6BT, United Kingdom (E-mail:    {\it j.ruf@ucl.ac.uk}).
          }
                                      }

\begin{document}

\maketitle

\centerline{\it To the Memory of Marc Yor  
}

 \medskip

\begin{abstract}
\noindent
We study the distribution of the time to explosion for one-dimensional diffusions.  We relate this question to the  computation of expectations of  suitable nonnegative local martingales.  Moreover, we characterize the distribution function of the time to explosion as the minimal solution to a certain Cauchy problem for an appropriate parabolic differential equation; this  leads to alternative characterizations of Feller's  criterion for explosions. We discuss in detail several examples  for which it is possible to obtain analytic expressions for the corresponding distribution of the time to explosion,  using the methodologies developed in the paper.
\end{abstract}



\input amssym.def
\input amssym

\section{Introduction and summary}

Precise conditions for whether or not a one-dimensional   diffusion process explodes in finite time  have been developed, most notably  by William Feller. To the best of our knowledge, and rather surprisingly, the distribution of the explosion time  has rarely -- if at all -- been the subject of  investigation (a notable exception is the appendix in \citet{Elworthy2010}, particularly its Theorem~9.1.3(ii)).   With the present  work  we hope to help close this gap  or,  at the very least, to narrow it.  

In Section~\ref{S:essentials}, we recall relevant facts of the theory of one-dimensional diffusions.  
In Section~\ref{S transformation},  we recall and generalize a result of \citet{McKean_1969} that associates the distribution of the explosion time to the expectation of a related nonnegative local martingale.   In Section~\ref{S:analytic}, we study analytic properties of the distribution function such as continuity, strict positivity,  and full support.
In Section~\ref{S:DE}, we characterize the  tail  $ (t, \xi) \longmapsto \Prob_\xi (S>t)$ of the probability distribution function of the time-to-explosion $S$, viewed as a function of time $\,t\,$ and starting position $  \xi$, as the smallest nonnegative solution of an appropriate    partial differential equation of parabolic type. We also derive a similar characterization for the Laplace transform of this tail probability distribution function  in terms of an ordinary differential equation, and present   alternative characterizations of Feller's criterion for explosions. In Section~\ref{S:examples},  we provide several examples to illustrate the methodologies developed in Sections~\ref{S transformation}--\ref{S:DE}.  In Appendices~\ref{FT} and \ref{Doss}, we recall some useful facts regarding the   Feller test and  the  Lamperti  transformation. Finally, in Appendix~\ref{A:uniqueness}, we provide a technical uniqueness result for stopped   diffusions in natural scale.

Although not much work seems to  exist on the distribution of the time to explosion, there is a huge literature   on the computation of first-passage times by diffusions. We refer  to \citet{Pitman_Yor_2003} and Section~2 in \citet{Going_Yor} for several pointers to this literature. If a   boundary   is \emph{regular}, the diffusion can be extended beyond that boundary and the time to explosion  can be represented as a first-passage time  for a regular diffusion; however, for an \emph{exit} boundary such an extension is not possible (one-dimensional diffusions cannot explode at {\it natural} or {\it entrance} boundaries). We refer the reader to Section~16.7 in \citet{Breiman1992} for a classification of boundary behavior for one-dimensional diffusions.

\section{Some essentials of one-dimensional diffusions}  \label{S:essentials}

We fix an open interval $I = ( \ell, r)$ with $- \infty \le \ell < r \le \infty$ and consider the stochastic differential equation
\begin{equation}
\label{1}
\mathrm{d} X(t) = \mathfrak{s} (X(t)) \big[\, \mathrm{d} W (t) + \mathfrak{b} (X(t)) \,\mathrm{d} t \,\big], \qquad t \geq 0, \qquad X(0) = \xi\,,
\end{equation}
where $ \xi \in I $ and $ W(\cdot) $ denotes a Brownian motion.  We shall impose throughout the paper the following: 

\begin{standassume}  
\label{A2}
{\it The 	 functions $\,\mathfrak{b}: I \rightarrow \R\,$ and $\,\mathfrak{s} : I \rightarrow \R \setminus \{ 0 \}$ are   measurable and satisfy  
}
\begin{equation}
\label{2}
\int_K \left(\frac{1}{ \mathfrak{s}^2 (y)} + \left|\frac{\mathfrak{b} (y)}{ \mathfrak{s} (y)}\right|\right)  \dx y \,  <  \, \infty  \qquad \hbox{{\it for every compact set }}\, K \subset I.\qquad 
\end{equation}
  In other words, we shall be assuming   that both   $1 /  \mathfrak{s}^2 (\cdot)$ and the ``local mean/variance ratio'' function
\begin{equation}
\label{f}
\mathfrak{ f} (\cdot) := { \mathfrak{b}(\cdot) \over  \mathfrak{s} (\cdot)}= {  \mathfrak{b}(\cdot)   \mathfrak{s} (\cdot)\over \mathfrak{s}^2 (\cdot)}
\end{equation} are locally integrable over the interval $I$.  \qed
\end{standassume}

In the sequel we shall also need  the anti-derivative  
\begin{equation}
\label{F_anti}
F(\cdot) \,:=\, \int_c^{\cdot} \mathfrak{ f} (x) \dx x 
\end{equation}
  of the function  $\,\mathfrak{f}(\cdot)$, for some   arbitrary, but fixed, constant $c \in I$.

From the   arguments   in \citet{Engelbert_Schmidt_lecture, Engelbert_Schmidt_1991} or Theorem~5.5.15 in \citet{KS1},  the standing assumption implies that the stochastic differential equation \eqref{1} admits a weak solution, unique in the sense of the probability distribution and defined up until the ``explosion time"
\begin{equation}
\label{3}
S \,:=\,  \lim_{n \uparrow \infty}   S_n\,, \qquad S_n  \,:=\, \inf \big\{  t \ge 0: X(t) \notin ( \ell_n, r_n) \big\},
\end{equation}
 for any monotone sequences $\{ \ell_n\}_{n \in \N}$, $\{ r_n\}_{n \in \N}$ with $ \ell < \ell_n < r_n < r$ and $\lim_{n \uparrow \infty} \ell_n = \ell$, $\lim_{n \uparrow \infty} r_n =r$.
The endpoints of the interval $I= ( \ell, r)$ are absorbing for   $X(\cdot)$, that is, we suppose that on $ \{ S < \infty\}$ this process stays after $S$ at the endpoint where it exits the interval.   Feller's test of explosions yields an analytic characterization whether such an explosion occurs; see Appendix~\ref{FT} for a review.   

We shall suppose that this weak solution has been constructed on some filtered probability space $ (\Omega, \F, \Prob)$, $\mathbb{F}= \{ \F (t) \}_{0 \le t < \infty}$ which satisfies the usual conditions of right-continuity and augmentation by $  \Prob-$null sets. Because of uniqueness in the sense of the probability distribution, the state process $X(\cdot)$ in this solution has the strong Markov property.   A thorough exposition and study of such equations appears in the monograph by \citet{Cherny_Engelbert}.

\subsection{A diffusion in natural scale} 
\label{sec1.2}

Let $X^o(\cdot)$ denote the state process of a solution to  the equation \eqref{1} without drift but with the same state space $   I = (\ell, r)\,$, that is, 
\begin{equation}
\label{13}
X^o (\cdot)  = \xi + \int_0^{ \cdot}  \mathfrak{s} \big( X^o (t) \big) \, \dx W^o (t)\,,
\end{equation}
where $W^o(\cdot)$ denotes a Brownian motion.
Thanks to our assumptions on   $  \mathfrak{s} (\cdot)$, this equation admits on some filtered probability space $ (\Omega^o, \F^o, \Prob^o),$ $\mathbb{F}^o= \{ \F^o (t) \}_{0 \le t < \infty}$ a weak solution $\,(X^o(\cdot), W^o(\cdot))\,$ which   is unique in the sense of the probability distribution \citep[see Theorem~5.5.7 in][]{KS1} up until an explosion time $\,S^o\,$, and has no absorbing points in the interval $I$ \citep[see Corollary~4.20 in][]{Engelbert_Schmidt_1991}. Here 
$\,S^o\,$, $S^o_n$ denote the stopping times   defined as in \eqref{3} with $X(\cdot)$ replaced by $X^o(\cdot)$.  We recall from Theorem~5.5.4  in \citet{KS1}  that $\,\Prob^o (S^o= \infty)=1\,$ holds if $\ell =-\infty$ and $r = \infty\,$.  Once again, the endpoints of the interval $I= ( \ell, r)$ are absorbing for $X^o(\cdot)$, that is, we suppose that on $ \{ S^o < \infty\}$ the process stays after $S^o$ at the endpoint where it exits.

\smallskip
In order to fix ideas and notation for what follows, let us summarize the construction of  this solution. We start with $ \xi \in I$ and a standard Brownian motion $B (\cdot)$, and define the stopping time 
\begin{equation}
\label{TAU}
 \taub  \,:=\,  \lim_{n \uparrow \infty}   \taub_n\,, \qquad  \taub_n \,:=\,  \inf \big\{  t \ge 0: \xi + B(t) \notin ( \ell_n, r_n) \big\}
\end{equation}
in the spirit and the notation of \eqref{3}. We introduce also  the  time change $$\Gamma (\theta) := \int_0^{\theta}  \frac{\dx r}{ \mathfrak{s}^2 ( \xi + B (r))} \,, \qquad 0 \le \theta < \taub
$$
with $\Gamma(\theta) := \infty$ for all $\theta \in [\taub,\infty)$. Next, we construct the inverse $ A (\cdot)$ of $ \Gamma(\cdot)$, and from it the state process $ X^o (\cdot) \equiv  X^o (\cdot\,; \xi) := \xi + B(A (\cdot))$. Finally, we construct the Brownian motion
$$
W^o(t) \,:=\, \int_0^{A  (t)} { \dx B (r) \over  \mathfrak{s}  ( \xi + B (r))} ,\qquad 0 \le t < S^o = \Gamma(\taub -)
$$
and check that the pair $ (W^o(\cdot), X^o (\cdot))$ satisfies \eqref{13}  up until the explosion time
\begin{equation}
\label{115}
S^o  = \Gamma (\taub -) = \int_{ 0}^{ \taub } { \dx \theta \over  \mathfrak{s}^2 ( \xi + B  (\theta))} = \int_I { 2 L (\taub, y)  \over \mathfrak{s}^2 (y)} \dx y\,;
\end{equation}
here $L (\taub, y)$ is the local time accumulated up to $ \taub$ by the Brownian motion $\xi + B(\cdot)$ at the site $y \in \R$.

\subsection{Transformation of scale}  
\label{SS scale}

Under the   conditions of  \eqref{2} and with the notation of  \eqref{f},  \eqref{F_anti}, the   diffusion $\, X(\cdot)\, $ of  \eqref{1} has {\it scale function}
\begin{equation}
\label{scale}
\mathfrak{p} (x) \,:=\, \int_c^x \exp \left( - 2  F (z) \right)  \mathrm{d} z\,.
\end{equation}
This   is a strictly increasing,   continuously differentiable bijection of the interval $\, I = ( \ell, r)\,$ onto the interval $\, J = \big( \lambda, \varrho)\,$ with endpoints $\, \lambda := \mathfrak{p} (\ell +)\,$ and $ \varrho := \mathfrak{p} (r-)$. We denote the inverse mapping $\, \mathfrak{q} := \mathfrak{p}^{-1}$ and check that   $\, \Upsilon (\cdot) := \mathfrak{p} \big( X(\cdot) \big)\,$ is a diffusion  in natural scale, with state space $J$, dynamics
\begin{align} \label{eq:upsilon}
	\Upsilon(\cdot) \,= \,\mathfrak{p}(\xi) + \int_0^{\, \cdot} {\bm \sigma}   \big(\Upsilon(t)\big) \,\dx W(t) 
\end{align}
up until the (same) explosion time $S$, and dispersion function
$\,
{\bm \sigma} (y)  :=  \big( \mathfrak{p}^\prime \cdot \mathfrak{s} \big) ( \mathfrak{q} (y)  )\,, ~ y \in J\,$; 
see, for instance,  Section~5.5B in \citet{KS1}. It is clear from this reduction that the explosion time $S$ of $X(\cdot)$ can be represented in the form \eqref{115}, namely
\begin{equation*}
S \,  = \int_ 0^ {\bm \zeta}  { \dx \theta \over   {{\bm \sigma}}^2 ( \mathfrak{p}(\xi) + B  (\theta))}\, = \int_J { 2 L ({\bm \zeta}, y)  \over  {\bm \sigma}^2 (y)} \dx y \,,
\end{equation*}
in terms of a standard Brownian motion $B(\cdot)$, its local time random field $\,L(\cdot\,,\cdot)$, and the first exit time 
\begin{align} \label{eq:bm zeta}
{\bm \zeta}  \,:=\, \lim_{n \uparrow \infty} 	{\bm \zeta}_n \, , \qquad  {\bm \zeta}_n \,:=\, \inf \big\{\theta \geq 0: \mathfrak{p}(\xi) + B(\theta) \notin (\lambda_n, \varrho_n)\big\} 
\end{align}
 for any monotone sequences $\,\{ \lambda_n\}_{n \in \N}\,$, $\{ \varrho_n\}_{n \in \N}\,$ satisfying  $ \,\lambda < \lambda_n < \varrho_n < \varrho\,$ and $\,\lim_{n \uparrow \infty} \lambda_n =\lambda\,$, $\lim_{n \uparrow \infty} \varrho_n =\varrho\,$. For later reference, we also introduce the time change
$$
\Gamma^{\Upsilon} (\theta) := \int_0^{\theta}  \frac{\dx r}{ {\bm \sigma}^2 ( \xi + B (r))} \,, \qquad 0 \le \theta < S
$$
with $\,\Gamma^{\Upsilon}(\theta) := \infty\,$ for all $\,\theta \in [S,\infty)$,   construct the inverse $ A^\Upsilon (\cdot)$ of $\, \Gamma^\Upsilon(\cdot)\,$, and note the representation $ \Upsilon (\cdot) \equiv  \mathfrak{p}(\xi) + B(A^\Upsilon (\cdot))$ for some standard Brownian motion $B(\cdot)$.

\section{Relating explosions to the martingale property}
\label{S transformation}
 
 \subsection{A generalized Girsanov theorem}
\label{GenGir}

We present here a generalized version of the Girsanov-Van Schuppen-Wong theorem, which appeared in Section~3.7 of \citet{McKean_1969} under conditions considerably stronger than those imposed here; see also Exercise~5.5.38 in \citet{KS1} for the special case $ \mathfrak{s} (\cdot) \equiv 1$ and $\ell =-\infty$, $r = \infty$ and Theorem~9.1.3(ii) in \citet{Elworthy2010}.   This version can be considered a ``weak'' result, as it  provides a distributional identity; \citet{Ruf_Novikov} and \citet{Ruf_Larsson} use a related ``strong'' (i.e., pathwise)  version to provide a proof of the sufficiency of the  Novikov and Kazamaki criteria for the martingale property of stochastic exponentials; see also \citet{Perkowski_Ruf_2014}.  

First, we recall the finiteness of integral functionals under additional square-integrability assumptions on certain related functions.
\begin{rem}{\it Finiteness of integral functionals.}
Let us assume that the local mean/variance ratio function $\,\mathfrak{f}(\cdot)\,$ is locally square-integrable on $I$. Furthermore, for fixed $T>0$,   let us denote by $\Lambda^X (T, y)$ the local time accumulated during the time interval $[0,T]$ by the semimartingale $ X(\cdot)$ in \eqref{1} at the site $y \in I$; for the properties of this random field,  see for example Theorem~3.7.1  in \citet{KS1}. 
  From the occupation time density formula in that   theorem, we have on   $ \{ S_n >T\}$ the $\Prob-$a.e. property
\begin{align}
\int_0^T \mathfrak{b}^2 (X(t))  \dx t  \,&=  \int_0^T \left( { \,\mathfrak{b} \,\over   \mathfrak{s}} \right)^2 ( X(t)) \dx \langle X \rangle (t) =\int_0^T \left(  \, \mathfrak{f} \,  \mathbf{ 1}_{ [ \ell_n, r_n) }\right)^2 ( X(t)) \dx \langle X \rangle (t) \nonumber \\
&= 2 \int_{  \ell_n}^{ r_n}   \mathfrak{f}^{2} (y)   \Lambda^X (T, y)\dx y  \le  2  \sup_{\ell_n \leq y  \le r_n} \big( \Lambda^X (T, y)\big) \cdot \int_{  \ell_n}^{ r_n}   \mathfrak{f}^{2} (y)    \dx y  <  \infty\,,
\label{5} 
\end{align}
where the last inequality follows from the c\`adl\`ag property of the function $\Lambda^X(T,\cdot)$.      
\qed
\end{rem}

We are now ready to state and prove a first result. For its purposes, we shall need    the Borel $\sigma-$algebra ${\cal B}$ generated by the open sets in $C ([0, \infty))$,    the mappings $\,\varphi_t\,$ defined as $\,(\varphi_t \mathfrak{w}) (s) := \mathfrak{w} (s\wedge t)\,, ~~ 0 \le s < \infty\,$, and the corresponding $\, \sigma-$algebras $\,{\cal B}_t:= \varphi^{-1}_t  ( {\cal B}  )$,   for all $t \in [0, \infty)\,$. In this vein, see Problem~2.4.2 in \citet{KS1}.

\begin{thm}{\bf Generalized Girsanov theorem.}
 \label{Thm1}
Suppose that the local mean/variance ratio function $\,\mathfrak{f}(\cdot)\,$ is locally square-integrable on $I$.
For any given $T \in (0, \infty)$ and any Borel set $\,\Delta \in {\cal B}_T 
\,$, we then have 
\begin{align} \label{18}
	\Prob \big( X(\cdot) \in \Delta, \, S>T \big) &= \E^o \left[ \exp \left( \int_0^T \mathfrak{b} \left( X^o (t) \right) \dx W^o (t) - \frac{1}{\,2\,} \int_0^T \mathfrak{b}^2 \left( X^o (t) \right) \dx t  \right)  \, \mathbf{ 1}_{\{X^o(\cdot) \in \Delta,\, S^o >T\}} \right].
\end{align}
In particular, if both diffusions $X(\cdot)\,$, $X^o(\cdot)$ are non-explosive, i.e., if $\, \Prob (S= \infty) = \Prob^o (S^o= \infty)=1\,$, then the exponential $\,\Prob^o-$local martingale 
$$
\exp \left( \int_0^T \mathfrak{b} \left( X^o (t) \right) \dx W^o (t) - \frac{1}{\,2\,} \int_0^T \mathfrak{b}^2 \left( X^o (t) \right) \dx t  \right) \,, \qquad 0 \le T < \infty
$$
is a true $\,\Prob^o$--martingale.
\end{thm}

\begin{proof}
We fix $T \in (0, \infty)$ and a Borel set $\, \Delta \in {\cal B}_T \,   $. 
In addition to the stopping times of \eqref{3}, we consider the stopping times
\begin{equation}
\label{6}
T_n :=S_n \wedge \inf \left\{ t \ge 0:  \int_0^t \mathfrak{b}^2 ( X  (s) )  \dx s \ge n \right\}, \qquad n \in \N,
\end{equation}
as well as stopping times $S^o_n$ and $T^o_{n}$  defined in the same manner as in \eqref{3} and \eqref{6}, but now with $X(\cdot)$ replaced by $X^o(\cdot)$.

We note that \eqref{5} implies $\{S>T\} = \bigcup_{n\in \N} \{T_{n}>T\}$, modulo $\Prob\,$;  similarly, we have $\{S^o>T\}  = \bigcup_{n \in \N} \{T_{n}^o>T\}$, modulo   $\Prob^o$.  In conjunction with the monotone convergence theorem, these observations imply that, in order to prove \eqref{18}, it is sufficient to show  
\begin{align} \label{E toshowT2.1}
	\Prob \big( X(\cdot) \in \Delta,  T_{n}>T \big) &= \E^o \left[ Z^o(T)  \mathbf{ 1}_{\{X^o(\cdot) \in \Delta, \,T_{n}^o>T\}} \right]
\end{align}
for all $n \in \N$, where we have set 
\begin{align*}
	Z^o(\cdot) := \exp \left( \int_0^{\cdot \wedge T_{n}^o} \mathfrak{b} ( X^o (t) )  \dx W^o (t) -  \frac{1}{2} \int_0^{\cdot \wedge T_{n}^o} \mathfrak{b}^2 ( X^o (t) )  \dx t  \right).
\end{align*}

In the following we shall prove  \eqref{E toshowT2.1} for fixed $n \in \N$.  Towards this end, we define the processes
$$
\widehat{W} (\cdot)  := \int_0^{\cdot \wedge T_{n}} \mathfrak{b} ( X(t))  \dx t \,+ \,W(\cdot)\,, \qquad 
	Z(\cdot) := \exp \left( \int_0^{\cdot \wedge T_{n}} \mathfrak{b} ( X (t) )  \dx\widehat{W} (t) -  \frac{1}{2} \int_0^{\cdot \wedge T_{n}} \mathfrak{b}^2 ( X (t) )  \dx t  \right)
$$
and the $\Prob$-local martingale
\begin{align*}
	L(\cdot) &:= \exp \left(- \int_0^{\cdot \wedge T_{n}} \mathfrak{b} ( X (t) )  \dx W (t) - {1 \over 2} \int_0^{\cdot \wedge T_{n}} \mathfrak{b}^2 ( X (t) )  \dx t  \right) = \frac{1}{Z(\cdot)}\,.
\end{align*}
 We note that $L(\cdot)$ is a strictly positive martingale; cf.~Corollary~3.5.13 in \citet{KS1}.  Therefore, $\dx \Q = L(T) \dx \Prob$ defines a new probability measure $\Q$ on $ (\Omega, \F(T))$.   Girsanov's theorem \citep[see Theorem~3.5.1 in][]{KS1} yields that $\widehat{W} (\cdot)$ is a $\Q$-Brownian motion.   Moreover, $\widehat{X}(\cdot) := X(\cdot \wedge T_{n})$ is a solution of the stochastic integral equation
  \begin{equation*}
\widehat{X}  ( \cdot)=  \xi + \int_0^{\,\cdot \wedge T_{n}}  \mathfrak{s} \big( \widehat{X} (t) \big)  \, \dx \widehat{W}  (t).
 \end{equation*}
Proposition~\ref{P:uniqueness appendix} implies that the $\Q$--distribution of $X(\cdot \wedge T \wedge T_n) $ is the same as the $\Prob^o$--distribution of $X^o(\cdot \wedge T \wedge T_n^o)$.  
Finally, we note that 
\begin{align*}
Z(\cdot) &= \exp \left( \int_0^{\,\cdot  \wedge T_{n}} \mathfrak{f} ( X (s) )  \dx X (s) -  \frac{1}{2} \int_0^{\,\cdot \wedge T_{n}} \mathfrak{b}^2 ( X  (s) )  \dx s  \right), \\
Z^o(\cdot) &= \exp \left( \int_0^{\,\cdot \wedge T_{n}^o} \mathfrak{f} ( X^o (s) )  \dx X^o (s) -  \frac{1}{2} \int_0^{\,\cdot \wedge T_{n}^o} \mathfrak{b}^2 ( X^o (s) )  \dx s  \right) 
\end{align*}
are also nonanticipative functionals of $X(\cdot \wedge T_{n})$ and $X^o(\cdot \wedge T^o_{n})$, respectively, so  we  have
\begin{align*}
	\Prob \big( X(\cdot) \in \Delta,  T_{n}>T \big) &= \E \left[L(T) \cdot Z(T) \mathbf{1}_{\{X(\cdot) \in \Delta,\, T_{n}>T\}} \right]  \\
		&= \E^{\Q} \left[Z(T) \mathbf{1}_{\{X(\cdot) \in \Delta, \,T_{n}>T\}} \right] = \E^o \left[ Z^o(T) \mathbf{ 1}_{\{X^o(\cdot) \in \Delta, \,T_{n}^o>T\}} \right] .
\end{align*}
This yields \eqref{E toshowT2.1} and concludes the proof.
\end{proof}

\subsection{Feynman-Kac representation}  
\label{SS feynman}

Fixing $T\in (0, \infty)\,$ and   taking $\Delta = C([0,T])$ in \eqref{18}, we obtain the distribution of the explosion time $S$ in \eqref{3} as  
\begin{equation}
\label{16}
\Prob (   S>T)= \E^o \left[ \exp \left( \int_0^T \mathfrak{b} \big( X^o (t) \big) \, \dx W^o (t) - {1 \over \,2\,} \int_0^T \mathfrak{b}^2 \big( X^o (t) \big) \,\dx t  \right)  \mathbf{1}_{\{S^o>T\}} \right],
\end{equation}
whenever the function $\,\mathfrak{ f} (\cdot)\,$ in \eqref{f} is locally square-integrable (thus also locally integrable) on   $I$. 

If we assume, in addition,  that $\mathfrak{ f} (\cdot)$ is also of finite first variation on compact subintervals of $I$ and left-continuous,  we have in the notation of \eqref{F_anti}   the generalized It\^o-Tanaka formula   
$$\,F( X^o (T)) - F(\xi)\,= \int_0^T \mathfrak{b}\big( X^o (t)\big)\,  \dx W^o (t) + \int_I \Lambda^{X^o} (T,a)\,\mathrm{d} \mathfrak{ f} (a)\,$$ 
on the event $\{ S^o > T\}$, where $\, \Lambda^{X^o} (T,a)\,$ denotes the semimartingale local time accumulated by $\, X^o (\cdot)\,$ at the site $\, a \in I$ during the time-interval $[0,T]$, and the expression \eqref{16} becomes 
\begin{equation}
\label{Ito-Tanaka}
\Prob(   S>T)=     \E^o \left[ \exp \left( F ( X^o (T) ) - F (\xi) -  \int_I \Lambda^{X^o} (T,a)\,\mathrm{d} \mathfrak{ f} (a) - {1 \over \,2\,} \int_0^T \mathfrak{b}^2 \big( X^o (t) \big) \, \dx t  \right)      \1_{ \{ S^o > T\} } \right].~~~~~~~~
\end{equation}

  Let us   assume next, that the function $\mathfrak{f}(\cdot)$ is actually    continuously differentiable on $I$; then  \eqref{Ito-Tanaka} takes the more ``classical" form
\begin{equation}
\label{17}
\Prob(   S>T)= \exp\big(- F (\xi)\big) \cdot   \E^o \left[ \exp \left( F ( X^o (T) )  -   \int_0^T V ( X^o (t) )  \dx t  \right)    \cdot \1_{ \{ S^o > T\} } \right] 
\end{equation}
with the notation
\begin{equation*}
V(x)\, := \, \frac{1}{\,2\,}  \, \mathfrak{s}^2 (x)   \left(  \mathfrak{f}^{2} (x) + \mathfrak{ f}^{\prime} (x)  \right) =  \frac{1}{\,2\,} \left( \mathfrak{b}^2 (x) + \mathfrak{ b}^{\prime} (x) \mathfrak{s} (x) -  \mathfrak{ b} (x) \mathfrak{s}^{\prime} (x)  \right), \qquad x \in I. 
\end{equation*}
In other words, the distribution of the explosion time is determined then completely by the joint distribution  of $X^o (T)$ and $ \int_0^T V \big( X^o (t) \big) \, \dx t\,$ on the event $ \{ S^o > T\} $, for all $T \in (0, \infty)$.

\begin{rem}{\it Non-explosive $X^o(\cdot)$.}  
 \label{Rem1}
 When $\Prob^o (S^o = \infty) =1$, 
the expression of  \eqref{17} takes the simpler form 
  \begin{equation*}
\Prob(   S>T )\,= \,\exp \big(- F (\xi)\big) \cdot   \E^o \left[ \exp \left( F \big( X^o (T) \big)  -   \int_0^T V \big( X^o (t) \big) \,\dx t  \right) \right]   .
\end{equation*} 
 In the   special case $  \mathfrak{s} (\cdot)  \equiv 1$   we have $ X^o(\cdot) = \xi + W^o(\cdot)$, so finding the distribution of the explosion time $S$ as in \eqref{17} amounts then to computing the joint distributions of appropriate Brownian functionals. \qed 
\end{rem}

\begin{rem}{\it Non-explosive $X (\cdot)$.} 
Of course, the reverse situation also prevails: when $\Prob  (S  = \infty) =1$  and the function $\mathfrak{f}(\cdot)$ is      continuously differentiable on $I$, the distribution of the explosion time of the diffusion $\,X^o(\cdot)\,$ in natural scale \eqref{13} is given as  
 \begin{equation*}
\Prob^o (   S^o>T )\,= \,\exp \big(  F (\xi)\big) \cdot   \E  \left[ \exp \left( -F \big( X  (T) \big)  +   \int_0^T V \big( X  (t) \big) \,\dx t  \right) \right]   .
\end{equation*} 
This can be argued in exactly the same manner. \qed
\end{rem}

\section{Analytic properties of the explosion time distribution}  \label{S:analytic}
In this section, we shall discuss analytic properties of the function $U: (0,\infty) \times I \rightarrow [0,1]$, defined via
\begin{equation}
\label{U}
U (T, \xi)  := \Prob_{\xi}( S>T ), \qquad (T, \xi) \in (0, \infty) \times I.
\end{equation}
Here and in what follows, we   index  the probability measure by the common starting position $\, \xi \in I$ of the diffusions $ X(\cdot)$ and $ X^o(\cdot)$. 
 
\subsection{Continuity}
\label{cont}

The question of continuity of the function $U(\cdot\,, \cdot)$ is of interest in itself; it also will be important for our arguments later on.
Since $1-U(\cdot, \xi)$ is a distribution function, it is   right-continuous for all $\xi \in I$.  In this subsection  we shall see, without any further assumptions on the coefficients $\mathfrak{s}(\cdot)$ and $\mathfrak{b}(\cdot)$ beyond those of Standing Assumption~\ref{A2},   that $U(\cdot\,, \cdot)$ is even jointly continuous in its two arguments. 

We   start with a technical result in Lemma~\ref{L min max}.  In particular, the property in \eqref{mystery} is well known for regular, one-dimensional diffusions; it is discussed, for instance, in the ``matching numbers" Section~3.3 of \citet{ItoMcKean}. It is not hard to prove from first principles, so we present here a simple argument.  We then show in Lemma~\ref{L t continuity} the continuity of the function $U(\cdot\,, \cdot)$ in the first component, as a function of time only. Finally, in Proposition~\ref{P continuity} we   establish the joint continuity of $U(\cdot, \cdot)$.

\begin{lem}{\bf Diffusions hit nearby points fast.} 
 \label{L min max}
	With the stopping times
	\begin{align} 
		H_x &:=  \inf \left\{t \in [0, \infty): X(t) = x\right\}, \qquad x \in I,  \label{E H}\\
		\widehat{H}_{x,y} &:= \inf \left\{t \in [H_{x}, \infty): X(t) =y\right\},\qquad (x,y)  \in I^2, \label{E Hhat}
	\end{align}
for any given $\varepsilon > 0$ there exist $x_1 \equiv x_1(\varepsilon) \in (-\infty, \xi)$ and  $x_2 \equiv x_2(\varepsilon) \in (\xi, \infty)$ such that
\begin{align}  \label{E widehat epsilon}
	\Prob_\xi \big(\widehat{H}_{x_i,\xi} < \varepsilon\big) \ge 1-\varepsilon, \qquad i =1,2.
\end{align}
In particular, for all $\delta>0$, we have
\begin{equation}
\label{mystery}
\lim_{y \rightarrow \xi}  \Prob_{y}\left(H_\xi < \delta \right) =1.
\end{equation}
	\end{lem}
\begin{proof}
	We first show that we have $\Prob_\xi(\mathcal{A}) = 0$ for the event
	\begin{align}  \label{E mathcal A}
		\mathcal{A} &:= \left\{ \omega \in \Omega: \exists \, R(\omega)>0 \text{ such that } \min_{t \in [0,R(\omega)]}  X(t,\omega)  \geq \xi \right\}.
	\end{align}
It is sufficient to show that \eqref{E mathcal A} holds with 
 the process $\,\Upsilon(\cdot)=\mathfrak{p} \big( X(\cdot) \big)\,$ of \eqref{eq:upsilon} instead of the diffusion $X(\cdot)$, and with $\, \xi\,$ replaced by $\,\mathfrak{p}(\xi)\,$, 
 due to the strict monotonicity of the scale function $\,\mathfrak{p}\,$ in  \eqref{scale}.  The path properties of standard Brownian motion, in conjunction with the representation  $ \Upsilon (\cdot) \equiv  \mathfrak{p}(\xi) + B(A^\Upsilon (\cdot))$ of Subsection~\ref{SS scale}  for some standard Brownian motion $B(\cdot)$,   and with the fact that $A^\Upsilon(t) > 0$ holds  for all $t> 0$,  let us conclude.

The continuity (from below) of the probability measure $\Prob_\xi$ then yields  the existence of $y_1 \in (-\infty, \xi)$ such that 
$\Prob_{\xi}(H_{y_1} < \varepsilon ) \ge 1-\varepsilon/2$.   Replacing the minimum by a maximum in \eqref{E mathcal A} and repeating the argument, we obtain the existence of $x_2 \in (\xi, \infty)$ such that
$\Prob_{\xi}(H_{x_2} < H_{y_1} < \varepsilon ) \ge 1-\varepsilon$ holds;    this then implies \eqref{E widehat epsilon} for $i=2$.  The existence of the claimed $x_1 \in (-\infty, \xi)$  is argued in the same manner. Finally, the strong {Markov} property of the diffusion $X(\cdot)$ implies  
\begin{align*}
	\Prob_{y}\big(H_\xi <  \varepsilon\big) \geq 
	\Prob_{\xi}\big(\widehat{H}_{y,\xi}  <  \varepsilon \big); 
\end{align*}
that is, the probability of the event that the diffusion $X(\cdot)$ started at $y$ hits $\xi$ before time $\varepsilon$  dominates the probability of the event that $X(\cdot)$ completes a round-trip from $\xi$ to $y$ and then back to $\xi$,  before time $\varepsilon$. We now fix  $\delta>0, {\varepsilon} \in (0, \delta)$ and the corresponding $x_1 \in (-\infty, \xi)$ and $x_2 \in (\xi, \infty)$. Then for all $y \in (x_1, x_2)$, applying \eqref{E widehat epsilon},  we have 
\begin{align*}
	\Prob_{y}\left(H_\xi <  \delta \right)  \geq  \Prob_{y}\left(H_\xi <  \varepsilon\right)  \geq  1-\varepsilon\,.
\end{align*}
The  proof of \eqref{mystery} follows.
\end{proof}

\begin{lem}{\bf Continuity of $U(\cdot, \cdot)$ as a function of time.}  \label{L t continuity}
	The function $T \mapsto U(T, \xi)$ is continuous 
	on $[0, \infty)$, ~for any given $\xi \in I$. 
\end{lem}
\begin{proof}
	We fix $(T,\xi) \in (0,\infty) \times I$ and observe that it is sufficient to show   $\,p := \Prob_\xi(S=T) = 0\,$, due to the   right-continuity of the function $U(\cdot\,, \xi)$.	Let us consider any strictly increasing sequence of stopping times $0 = H^{(0)} < H^{(1)} < H^{(2)} < \cdots$.  We then have, again by the strong Markov property of the diffusion $X(\cdot)$, the comparison 
\begin{align*}
	1 &\geq \Prob_\xi\left(S \in \left\{T + H^{(i)}: i \in \N_0\right\}\right) = \sum_{i \in \N_0}   \Prob_\xi\left(S - H^{(i)} = T\right)\\ &\geq \sum_{i \in \N_0}  \E_\xi\left[ \, \Prob_{X(H^{(i)})} \big(S = T\big) \, \1_{\{X(H^{(i)}) \in I\}} \, \right]
		\, \geq \, p \sum_{i \in \N_0}  \Prob_\xi   \left(X\big(H^{(i)}\big) = \xi\right).
\end{align*}
Thus, in order to show the statement, it is sufficient to construct a strictly increasing sequence of stopping times $\{H^{(i)}\}_{i \in \N}$ such that $\Prob_\xi\left(X(H^{(i)}) = \xi\right)$ does not converge to zero as $i$ increases. We shall construct such a sequence inductively, by ``stitching together'' the round trips of Lemma~\ref{L min max}. Towards this end, consider a sequence $\{q_i\}_{i \in \N} \subset I$   such that \eqref{E widehat epsilon} holds with $x_1$ replaced by $q_i$ and $\varepsilon$ replaced by $\,\varepsilon_i \in (0,1)$ such that the series $\, \sum_{i \in \mathbb{N}} \log(1 -\varepsilon_i)$ converges.  Next,   define the stopping times
\begin{align*}
	\widetilde{H}^{(i)}  := \inf \left\{t \in [H^{(i-1)}, \infty): X(t) =q_i\right\}\,, \qquad 
	{H}^{(i)}  := \inf \left\{t \in [\widetilde{H}^{(i)}, \infty): X(t) =\xi\right\} 
\end{align*}
and use the Markov  property of   $X(\cdot)$, along with conditioning on the event $\, \big\{X(H^{(i-1)}) = \xi\big\}\,$, to obtain 
\begin{align*}
	\Prob_\xi\left(X(H^{(i)}) = \xi\right) &\geq \Prob_\xi\left(X(H^{(i-1)}) = \xi\right) \cdot \Prob_\xi\left(\widehat{H}_{q_i,\xi} < \varepsilon_i\right) \geq \Prob_\xi\left(X(H^{(i-1)}) = \xi\right) \cdot  \big(1 - \varepsilon_i \big)\\
	&\geq \ldots \geq \prod_{j=1}^i  \Big(1 - \varepsilon_j \Big) \,=\, \exp\left(\sum_{j=1}^i \, \log\left(1 -\varepsilon_j\right) \right),
\end{align*}
which does not tend to zero as $\,i\,$ increases. This concludes the proof.
\end{proof}

\begin{prop}{\bf Joint continuity of $\,U(\cdot, \cdot)\,$.}  \label{P continuity}
	The function $(T, \xi) \mapsto U(T, \xi)$ is jointly continuous  on $[0,\infty) \times I$.
\end{prop}
\begin{proof}
	We fix a pair $(T,\xi) \in [0,\infty) \times I$ and a sequence $\{(t_n,\xi_n)\}_{n \in \N} \subset [0,\infty) \times I$ such that $\lim_{n \uparrow \infty} (t_n,\xi_n) = (T,\xi)$.  

We start with the case $T = 0$. We need to show that $\lim_{n \uparrow \infty} U (t_n,\xi_n)= \lim_{n \uparrow \infty} \Prob_{\xi_n}(S>t_n) = 1\,$.    With $S = S(\ell) \wedge S(r)$, the minimum of the explosion times of $X(\cdot)$ to $\ell$ and $r$, respectively, let us  choose some $ \eta \in (\ell, \xi)$ and observe that, for sufficiently large $n\in \N$, we have the upper bound
\begin{align*}
	\Prob_{\xi_n} \big(S(\ell)> t_n\big) \geq \Prob_{\eta}  \big(S(\ell) > t_n  \big)=  U (t_n, \eta).
\end{align*}
This last quantity converges to $U(0, \eta) =1$  as $n$ tends to infinity
, due to the right-continuity of the function $U(\cdot, \eta)$. We obtain thus  $\,\lim_{n \uparrow \infty} \Prob_{\xi_n}(S(\ell)> t_n)=1$; similarly  $\lim_{n \uparrow \infty} \Prob_{\xi_n}(S(r)> t_n)=1$, and this proves the claim for $\,T=0$.

We   assume now $\,T>0$ and fix some $\varepsilon>0$. From Lemma~\ref{L t continuity}, there exists $\delta \in (0, T/2)$ so that $|U(t,\xi) - U(T,\xi)| < \varepsilon$ holds for all $t \in (T-2 \delta, T+2 \delta)$. Without loss of generality, we assume  $|t_n - T| <\delta $  for all $n \in \mathbb{N}$. Next, we observe that the strong {Markov} property of $X(\cdot)$ implies   
$$
	 \Prob_{\xi_n}\big({\{S>t_n\}}  \cap {\{H_\xi < \delta\}} \big) \,= \,\E_{\xi_n}\left[ U(t_n - H_\xi,\xi) \,\1_{\{H_\xi < \delta\}} \right],
$$
  thus
\begin{align*}
	\big|U(t_n,\xi_n) - U(T,\xi)\big| &= \left|\,
	 \E_{\xi_n}\left[ \left( \1_{\{S>t_n\}} - U(T,\xi) \right) \1_{\{H_\xi < \delta\}} \right] + \E_{\xi_n}\left[ \left( \1_{\{S>t_n\}} - U(T,\xi) \right) \1_{\{H_\xi \ge  \delta\}} \right]
	\, \right|
	                 \\
			&\leq \E_{\xi_n}\left[ \big|U(t_n - H_\xi,\xi) - U(T,\xi)\big| \, \1_{\{H_\xi < \delta\}} \right] +   \Prob_{\xi_n}\big(H_\xi \geq  \delta  \big)
			\\
			&\leq \varepsilon +   \Prob_{\xi_n} \big(H_\xi \geq   \delta \big)
\end{align*}
for all $n \in \N$. Here  $H_{\xi}\,$, defined as in \eqref{E H}, is  the first hitting time of $\xi$ by the process $X(\cdot)$; and we have noted that the inequalities $T - 2 \delta < t_n - H_\xi < T +\delta $  hold  on the event $\,\{H_\xi < \delta\}\,$ for all $n \in \N$.  Letting $n$ tend to infinity and applying \eqref{mystery} concludes the proof.
\end{proof}

\subsection{Strict positivity}
\label{SP}

We shall show in this subsection that the distribution of the explosion time $\, S \,$ in \eqref{3} cannot possibly have compact support.
\newpage
\begin{prop}{\bf The  explosion time distribution is not  supported on a compact set.}  
 \label{L:Drift1}
For all $\,K \in \R\,$,  we have $\,\Prob_\xi(S>K)>0\,$. 
\end{prop}

\begin{proof} 
The scale considerations in Subsection~\ref{SS scale} make clear that it is enough to consider diffusions in natural scale; so we shall prove $\,\Prob^o_\xi(S^o>K)>0\,$ for all $\,K \in (0,\infty)\,$. In the light of the representation \eqref{115},   setting $\, h := 1 / \mathfrak{ s}^2\,$  and recalling the stopping time $\, {\bm \tau}\,$ from \eqref{TAU}, it suffices then to show 
 \begin{equation}
 \label{ES}
 	\Prob^o_\xi \left(\int_0^{\,\taub} h\big(\xi + B(\theta) \big) \dx \theta > K\right)>0 \,, \qquad \forall ~~K \in (0,\infty)\,.
\end{equation}

We shall argue by contradiction, so let us suppose that \eqref{ES} fails; to wit,  that 
\begin{equation}
\label{contra}
\int_0^{\,\taub} h\big(B(\theta)+\xi\big)\, \dx \theta \,\leq K 
\end{equation}
holds $\,\Prob^o_\xi-$a.e. for some real constant $K>0$. To help obtain a contradiction, we first consider a diffusion $Y(\cdot)$ with state space $I = (\ell, r)$, solution of the stochastic equation
\begin{align*}
	Y(t) \,=\, \xi + \int_0^t \left(\frac{1}{Y(\theta)-\ell} - \frac{1}{r-Y(\theta)} \right) \dx \theta + W(t) \,, \qquad 0 \le t < \infty\,.
\end{align*}
This equation has a solution which is unique in the sense of the probability distribution; we also observe that the lifetime of this diffusion is $\, \mathcal{S}^Y = \infty\,$, that is, the endpoints of the interval $I = (\ell, r)$ are never reached. Next, we note that $\, 
	\int_0^\infty h\big(Y(\theta)\big) \,\dx \theta = \infty
\,$ 
holds almost surely,  by  Theorem~2.10(ii)  in \citet{MU_integral}; to apply this result, use $f(\cdot) = h(\cdot) \wedge 1 \leq h(\cdot)$.   Thus, there exists some $T>0$ such that 
\begin{equation}
\label{contra2}
\Prob^o_\xi \left(\int_0^T h(Y(\theta)) \dx \theta >K \right) >0\,.
\end{equation}
Denoting $\, B^{\taub} (\cdot) \equiv B (\cdot   \wedge \taub)\,$, we   define now a nonnegative local martingale $M(\cdot) $ as follows:  
 
\begin{enumerate}
 	\item[(i)] $M (\cdot) \equiv 1$, if $\ell = -\infty$ and $r = \infty$;
	\item[(ii)]  $M (\cdot) \equiv \big( \xi + B^{\taub }(\cdot)-\ell \big) \big/ (\xi - \ell)=\mathcal{E} \left( \int_0^{\,\cdot \wedge \taub}  (  \xi+  B(\theta)-\ell )^{-1}  \dx B(\theta) \right)  $, if $\ell >-\infty$ and $r = \infty$;
	\item[(iii)]  $M(\cdot) = \big(r-\xi-B^{\taub}(\cdot)\big) \big/ (r-\xi)=\mathcal{E} \left( \int_0^{\,\cdot \wedge \taub}  ( \xi+   B(\theta)-r)^{-1}  \dx B(\theta) \right)  $,  if $\ell =-\infty$ and $r < \infty$;
	\item[(iv)]  and finally,
\begin{align*}
	M(t) &= \left( \frac{\,\xi+ B^{\taub}(t)-\ell\,}{\xi - \ell} \right)\left( \frac{\,r-\xi-B^{\taub}(t)\,}{r-\xi}\right)\cdot \exp\left(\int_0^{t \wedge \taub} \frac{\dx \theta}{(\xi+B(\theta)-\ell )(r-\xi-B(\theta))} \right)\\
	&= \mathcal{E} \left( \int_0^{t \wedge \taub} \left(   \frac{1}{\xi+ B(\theta)-\ell } - \frac{1}{r-\xi-B(\theta)}\right) \dx B(\theta) \right) , \qquad 0 \leq t < \infty\,,
\end{align*}
 if $\,\ell >-\infty\,$ and $\,r < \infty\,$, 
 \end{enumerate}
where $\mathcal{E}(\cdot)$ denotes   stochastic exponentiation. 
In each of these four cases, the local martingale $M(\cdot) $ is a true martingale with expectation equal to $\, M(0)=1\,$.  This is obvious in the first three cases, and follows from the considerations of \citet{MU_martingale} or of \citet{Ruf_martingale} 
in the last case.
 
 \smallskip
 Thus, we may define a probability measure $\Q^{(T)}$ on $\mathcal{F}(T)$ via the recipe   $\dx \Q^{(T)} / \dx \Prob^o_\xi = M(T)$. We observe that the process $\xi + B(\cdot)  $ solves the same stochastic differential equation under this new  measure $\Q^{(T)}$, as the process $Y(\cdot)$ does under the measure $\Prob^o_\xi\,$, again in each of the four cases. Recalling that this stochastic differential equation has a solution which is unique in the sense of the probability distribution, as well as   \eqref{contra} and \eqref{contra2},  we obtain the contradiction 
 \begin{align*}
 	0 &< \Prob^o_\xi \left(\int_0^T h(Y(\theta)) \dx \theta >K \right) =\Q^{(T)} \left(\int_0^{ T }h(B(\theta)+\xi) \dx \theta >K \right)  	\\ &= \mathbb{E}^{\Prob^o_\xi} \left[M(T)\, \1_{\left\{\int_0^{ T }h(B(\theta)+\xi) \dx \theta >K\right\}}\right] \le  \mathbb{E}^{\Prob^o_\xi} \left[M(T)\, \1_{\left\{\int_0^{\taub   }h(B(\theta)+\xi) \dx \theta >K\right\}}\right]\,=\,0	\,.
	 \end{align*}
Here the last inequality follows from the fact that $M(T)  = 0$ holds on the event $\{\taub \leq T\}$.    
The statement is now proved.
\end{proof}

\begin{rem}{\it Alternative argument.}
	The referee suggests an alternative and shorter proof of Proposition~\ref{L:Drift1}: Let us recall the notation of  \eqref{E Hhat},  and 	observe that there exist some $\eta \in I$ and $\delta > 0$ such that $\Prob_\xi(\delta \leq \widehat H_{\eta, \xi} < S) > 0$.  By the strong Markov property this then yields that $\Prob_\xi(n \delta \leq \widehat H_{\eta, \xi} < S) > 0$ for each $n \in \N$, and thus  the statement.
	\qed
\end{rem}

\subsection{Full support}
\label{SS: FSP}

We shall show now that, when explosions can occur in finite time with positive probability,     Assumption~\ref{A1} below guarantees that the distribution of the explosion time has full support on the positive real half-line. Let us start by considering the closed set
\begin{align}
\label{eq:A}
	A \,:=\, \left\{x \in I: \int_{-\varepsilon}^{\varepsilon}\, \,\frac{\dx z}{\mathfrak{s}^4(x+z)} = \infty, ~~~\forall ~~\varepsilon \in \big(0, \min\{x-\ell, r-x\}\big)\right\}.
\end{align}

\begin{assume}  
\label{A1}
	{\it The set $A$ in \eqref{eq:A} is countable.}
\qed
\end{assume}

Every    $\,\mathfrak{ s}: I \to \R \setminus \{ 0 \}\,$ which is   continuous or, more generally, locally bounded away from the origin,  satisfies this assumption; for then the set $\,A\,$ is empty.  In the example that follows, we construct a non-trivial discontinuous function, not bounded away from zero locally,  that satisfies Assumption~\ref{A1}.
Thus, in most cases of interest,  Assumption~\ref{A1} is satisfied; however, as   Example~5.23 in \citet{Wise_1993}
illustrates, it is even possible to have $A = I$ despite the integrability condition in  \eqref{2}.

\begin{exmp}
We observe that the  intervals $$I_{n,m} = \left(\frac{1}{n+1} + \frac{1}{n(m+1)(n+1)}\,,\, \frac{1}{n+1} + \frac{1}{nm (n+1)}\right]\,, \qquad (n,   m) \in \N^2$$ are disjoint since $\bigcup_{m \in \N} I_{n,m} = (1/(n+1), 1/n]$ for each $n \in \N$ and satisfy
$\bigcup_{(n,   m) \in \N^2} I_{n,m} = (0,1]$.
Consider the function $f: (-1,1) \rightarrow (0, \infty)$ $($we interpret  $f(\cdot)  \equiv 1/ \mathfrak{s}^{4}(\cdot) )\,$,  defined by
$f(x) = 1$ for all $x \in (-1,0]$ and by
$$
f(x) =  \left( x-\frac{1}{n+1} - \frac{1}{n(m+1)(n+1)} \right)^{-1}\quad \text{for all}~~~x \in I_{n,m} \setminus\{1\}\,.
$$
  We now consider the state space $I = (-1,1)$, and note that 
\begin{align*}
	A = \left\{\frac{1}{n+1} + \frac{1}{n(m+1)(n+1)}\,:\,(n,   m) \in \N^2  \right\} \bigcup \left\{\frac{1}{n+1} : n \in \N\right\}  \bigcup \{0\};
\end{align*}
thus, Assumption~\ref{A1} is satisfied.
\qed
\end{exmp}

\begin{thm}{\bf Distribution function of time to explosion is strictly decreasing.}
 \label{PP}
For any fixed starting position $\, \xi \in I\,$, the function 
$$
[ 0 , \infty) \ni T \longmapsto \mathbb{P}_\xi \big( S  > T \big) \in (0, 1]
$$
is   strictly decreasing, provided that $\, \mathbb{P}_\xi \big( S  = \infty \big) <1\,$ and Assumption \ref{A1}   hold. 
\end{thm}

The argument will require a few preliminaries.  The next lemma contains the key idea in the proof of the main  result. It asserts that, under the local integrability of the reciprocal of the second power of its local variance function, the diffusion in natural scale $\, X^o (\cdot)\,$ ``can reach   far away points fast, with positive probability." 

\begin{lem}{\bf $X^o$ reaches far-away points fast,  with positive probability.}
\label{L:Integrability}
Assume  there exist $z \in (\ell,\xi)$ and $y \in (\xi, r)$ such that $\,\,
	\int_z^{y}  \mathfrak{s}^{-4}(a) 
	\, \dx a < \infty\,$. 
 Then for every $\,\varepsilon>0\,$ we have
\begin{align}  \label{eq:L1.2a}
	\Prob_{\xi}^o\left(H_{z} < \varepsilon\right) > 0\,.
\end{align}
\end{lem}

\begin{proof}
We start by fixing the constant  
$$K = \frac{(\xi-z)({y} -z)}{\varepsilon({y} -\xi)} + 1>0\,,$$ and considering the $\Prob^o_\xi$--local martingale 
	$$
	L(t)\,: =\, - K \int_0^{\,t \wedge H_z \wedge H_{y}} \frac{\dx W^o(u)}{\,\mathfrak{s}(X^o(u))\,} \, =\, -K \int_0^{\,t \wedge H_z \wedge H_{y} } \frac{\dx X^o(u)}{\,\mathfrak{s}^2(X^o(u))\,} \,, \qquad 0 \le t < \infty\,. 
	$$
	The stochastic integral here is well-defined, because
	$$
\int_0^{t \wedge H_z \wedge H_{y}}  \frac{\dx u}{\mathfrak{s}^2(X^o(u))} \, =\, \int_0^{t \wedge H_z \wedge H_{y}}  \frac{\dx \langle X^o\rangle(u)}{\mathfrak{s}^4(X^o(u))} \, =\,   \int_z^{y}   2\,\Lambda^{X^o}\big(t \wedge H_z \wedge H_{y},a\big)\, \frac{\dx a}{\mathfrak{s}^4(a)}\, < \,\infty
$$
holds for all $\, t \in [0, \infty)\,$. 	Here, the second equality follows from the occupation-time-density property of semimartingale local time, and the strict inequality   from the c\`adl\`ag property of the mapping $\,a \mapsto \Lambda^{X^o}(t \wedge H_z \wedge H_{y},a)\,$.
	
\smallskip
	Next, we consider the stochastic exponential $M = \mathcal{E}(L)$.  By \citet{MU_martingale}, the local martingale $M(\cdot)$ is a true martingale (use Theorem~2.1 in that paper,  with function $b(\cdot) = -K \1_{ [z,y]} (\cdot) /\mathfrak{s}(\cdot)$ along with (24) and  (26) there). Consequently, by the Girsanov theorem, 	there exist a probability measure $\,\Q_\xi^o\,$, absolutely continuous with respect to $\,\Prob_\xi^o\,$ on $\, {\cal F} (\varepsilon)\,$,
	and a $\,\Q_\xi^o$--Brownian motion $\widetilde{W}(\cdot)$, so that  $\,X^o(\cdot)\,$ satisfies up until the explosion time $S^o$ the 
	equation
\begin{align}  
 \label{eq:XoNewMeasure}
	X^o(\cdot) \,= \,\xi - K (\cdot\, \wedge H_z \wedge H_{y})  + \int_0^{\,\cdot} \mathfrak{s}\big(X^o(t)\big) \,\dx \widetilde{W}(t)\,.
\end{align}
	
	Now  let us assume that \eqref{eq:L1.2a} fails, 	that is, $\,	\Q_\xi^o\left(H_{z} \ge \varepsilon\right) = 1 \,$ holds for some $\, \varepsilon >0\,$; as a consequence,  the process $X^o (\cdot)$ is then bounded from below by $z$ on the time-interval $[0, \varepsilon]$. Thus,  we have   
$$
\int_0^{ t} \mathfrak{s}(X^o(u)) \dx \widetilde{W}(u)\,=\, X^o(  t) + K(  t \wedge H_{{y}}) - \xi\,, \qquad 0 \le t \le \varepsilon 
$$ 
$\Q_\xi^o-$a.e., and deduce that the process on the left-hand side is a $\Q_\xi^o$--local martingale, bounded from below by $z-\xi$ on the interval $[0, \varepsilon]$.  This again implies that  $X^o(\cdot  \wedge \varepsilon \wedge H_{{y}})$ is a $\Q_\xi^o$--supermartingale, and    
$$
\Q_\xi^o \left(H_{{y}} \geq \varepsilon\right) \,\geq \,\Q_\xi^o \left(H_{{y}} \geq H_{{z}}\right) \,\geq \,\frac{{y}-\xi}{{y}-z}\,.
$$
Taking expectations in \eqref{eq:XoNewMeasure}, we obtain
\begin{align*}
	z \,\leq \,\E^{\Q_\xi^o} \left[X^o\left(\varepsilon\right)\right] \,\leq \, \xi - K \varepsilon \cdot\Q_\xi^o\left(H_{{y}} \geq \varepsilon\right)\,
\leq \,\xi - K \varepsilon \cdot \frac{ \,  {y}-\xi \,}{ y-z}\, < z
\end{align*}		
by the definition of $K$.  This apparent contradiction yields \eqref{eq:L1.2a} and concludes the proof. 
\end{proof}

We continue  by showing that  under Assumption~\ref{A1}, when  the diffusion $X^o(\cdot)$ can explode  in finite time with positive probability, it can explode arbitrarily fast  with positive probability. Although the proof of the  lemma that follows is somewhat tedious, the underlying idea is quite simple.  First,   it   suffices  to show that the diffusion, started in $\xi$, can hit  a certain point $y \in I$ arbitrarily fast; let us assume, for the moment, that $y < \xi$. Then,    we choose some $K \in \N$ and decompose the interval $[y,\xi]$ in $2K-1$ intervals, each of which the diffusion $X^o(\cdot)$ can cross  fast enough with positive probability, so that the total time until the diffusion $X^o(\cdot)$ hits $y$ can  still be made arbitrarily small with positive probability.

  Among these $2K-1$ intervals, there are $K$ ones that contain all ``critical'' points in the set $A \bigcap [y,\xi]$ as in \eqref{eq:A}.  Their lengths are chosen so  that   the diffusion $X^o(\cdot)$ can cross each of  these  $K$ intervals fast enough; in the proof that follows, these intervals have the form $(x_{i_k}-\delta_{i_k}^\prime, x_{i_k}+\delta_{i_k})$. The remaining $K-1$ intervals, of the form $(x_{i_{k+1}}+\delta_{i_{k+1}}, x_{i_k}-\delta_{i_k}^\prime)$,   do not contain any points of the set $A$, and  may be quite large; nevertheless, an application of Lemma~\ref{L:Integrability} guarantees  that the diffusion can cross them fast enough with   positive probability. This   concludes the argument for diffusions in natural scale.

\begin{prop}{\bf If $X^o$ can  explode, it explodes fast with positive probability.} 
\label{L:noDrift2}
	Under Assumption~\ref{A1}, and provided   $\,\,\Prob^o_\xi(S^o<\infty) > 0\,$ holds, the explosion time of \eqref{115} satisfies  $$\Prob^o_\xi \big(S^o < \varepsilon \big)>0~~~~\hbox{ for all}~~\varepsilon > 0\,. $$  
\end{prop}
\begin{proof}
We fix $\varepsilon>0$ and note that there exists a point $y \in (\ell, r)$ such that $\Prob_{y}^o(S^o<\varepsilon/2)>0$ and   $\Prob_{\xi}^o(H_y < \infty)>0$ hold, where $H_y$ is defined in  \eqref{E H}; for otherwise, we would have  $\Prob^o_\xi(S^o < \infty) = 0$.  Thus, we need only    show     $\Prob_\xi^o(H_{y} <\varepsilon/2)>0$.  Without loss of generality, we shall assume   $y \in (\ell,\xi)$.

\smallskip
 We     enumerate now as $\{x_i\}_{i \in \N}$  the points of the closed set  $A \bigcup \{y,\xi\}$, where $A$ is given right before Assumption~\ref{A1}   (if      $A \bigcup \{y,\xi\}$ has only a finite number of   points, say $m \in \N$, we just set $x_{m+1} = x_{m+2} = \ldots = \xi$).   Fix a sequence of strictly positive numbers $\{\varepsilon_i\}_{i \in \N}$ such that $\sum_{i \in \N}  \varepsilon_i < \varepsilon/4$. As in the proof of Lemma  \ref{L min max},  there exist strictly positive numbers $\,\delta_i\,, \delta_i^\prime \, $ such that 
\begin{align*}
	\Prob^o_{x_i+ \delta_i}\big(H_{x_i-\delta_i^\prime} <   {\varepsilon_i}\big) > 0
\end{align*}
holds for each $i \in \N$.  An application of the Heine-Borel theorem then yields   the  existence of an integer $K \in \N$ and of $K$ points, say   $\xi = x_{i_1} > \cdots > x_{i_K} = y$, such that 
\begin{equation*}
\left(A \bigcup \{y,\xi\}\right) \bigcap \,[y,\xi] ~ \subset ~\bigcup_{k = 1}^K  \left(x_{i_k}-\frac{\,\delta_{i_k}^\prime\,}{2}\,,\, x_{i_k}+\delta_{i_k}\right) 
\end{equation*}
holds for the corresponding positive numbers $\delta_{i_k}$, $\delta_{i_k}^\prime$.  
We may assume    $\,x_{i_k} - \delta_{i_k}^\prime > x_{i_{k+1}} + \delta_{i_{k+1}}\,$   for all $k = 1, \ldots, K-1$; since if one of these last inequalities did not hold, we could just merge two overlapping intervals of the form $(x_{i_k}-\delta_{i_k}^\prime, x_{i_k}+\delta_{i_k})$ and $(x_{i_\ell}-\delta_{i_\ell}^\prime, x_{i_\ell}+\delta_{i_\ell})$ to one of the form  $(x_{i_k}-\widetilde{\delta}_{i_k}^\prime, x_{i_k}+\widetilde{\delta}_{i_k})$ (and replace the two corresponding $\varepsilon_i$'s by their sum) and repeat this procedure until all strict inequalities were made to hold.

If we   show now that the diffusion $X^o(\cdot)$ can move  fast through those intervals of the form $(x_{i_{k+1}}+\delta_{i_{k+1}}, x_{i_k}-\delta_{i_k}^\prime)$ with positive probability, that is, if the inequalities 
\begin{align}  
\label{eq:L2.2a}
	\Prob^o_{x_{i_k}-\delta_{i_k}^\prime}\left(H_{x_{i_{k+1}}+\delta_{i_{k+1}}} < \frac{\varepsilon}{4 K}\right) > 0, \qquad \forall ~~~k = 1, \ldots, K-1 
\end{align}
hold, then we obtain  from the Markov property
\begin{align*}
	\Prob^o_\xi \big(S^o < \varepsilon\big) &\geq \Prob^o_\xi\left(H_y < \frac{\varepsilon}{2}\right) \, \Prob^o_y\left(S^o < \frac{\varepsilon}{2}\right)
\geq  \Prob^o_{x_{i_1} + \delta_{i_1}}\left(H_{x_{i_K} - \delta_{i_K}^\prime} < \frac{\varepsilon}{2}\right) \,\Prob^o_y\left(S^o < \frac{\varepsilon}{2}\right)  > 0.
\end{align*}
Here, the second inequality holds since $\xi < x_{i_1} + \delta_{x_1}$ and $y>x_{i_K} - \delta_{i_K}^\prime$ by construction of the sequence $\,(x_i\,,\,i \in \N)\,$, and the last inequality holds since
\begin{align*}
	\Prob^o_{x_{i_1} + \delta_{i_1}}\left(H_{x_{i_K} - \delta_{i_K}^\prime} < \frac{\varepsilon}{2}\right)  &\geq  \prod_{k=1}^{K-1}  \left( \Prob^o_{x_{i_k} + \delta_{i_k}}\left(H_{x_{i_k} - \delta_{i_k}^\prime} < {\varepsilon}_{i_k}\right)   \Prob^o_{x_{i_k} - \delta_{i_k}^\prime}   \left(H_{x_{i_{k+1}} + \delta_{i_{k+1}}} < \frac{\varepsilon}{4K}\right)  \right) \\
	&~~~~~~~~~~~~~~~~~\cdot  \Prob^o_{x_{i_K} + \delta_{i_K}}\left(H_{x_{i_K} - \delta_{i_K}^\prime} < {\varepsilon}_{i_K}\right)   
	>0\,.
\end{align*}
This yields the statement of the proposition.  

In order to show \eqref{eq:L2.2a}, we fix $\,k \in 1, \ldots, K-1$ and note   $\,A \,\bigcap \, \left[ \, x_{i_{k+1}} + \delta_{i_{k+1}}  \,,\, x_{i_k} - \frac{\delta_{i_k}^\prime}{2} \,\right] = \emptyset\,$,  which implies  
\begin{align*}
	\int_{x_{i_{k+1}} + \delta_{i_{k+1}}}^{x_{i_k} - \frac{\delta_{i_k}^\prime}{2}} \frac{\dx a}{\mathfrak{s}^4(a)} < \infty 
\end{align*}
from the definition of the set $A$ and another application of the Heine-Borel theorem (yielding that any open cover of the compact interval $[x_{i_{k+1}} + \delta_{i_{k+1} } , x_{i_k} - \delta_{i_k}^\prime/2]$ has a finite subcover).
Thus, the assertion in \eqref{eq:L2.2a} follows now from Lemma~\ref{L:Integrability}, and this concludes the proof.
\end{proof}

Propositions \ref{L:Drift1} and   \ref{L:noDrift2} together yield the following statement:
\begin{prop}{\bf Distribution function of time to explosion of $X^o$ is strictly decreasing.}  \label{P:noDrift}
Under Assumption~\ref{A1}, the function $\,  [0,\infty) \ni t  \,\longmapsto \,\Prob^o_\xi(S^o>t) \in [0,1]\,$ is strictly decreasing, provided   $\,\Prob^o_\xi(S^o<\infty) > 0\,$ holds.
\end{prop}

\begin{proof}  We know from Proposition~\ref{L:noDrift2} that
\begin{equation}
\label{fast}
\mathbb{P}^o_y \big( S^o \le \varepsilon\big) \, > \, 0\,, \qquad \forall \varepsilon >0
\end{equation}
holds for any $y \in I$. For all $\, t \ge 0\,$ and $\,\varepsilon >0\,$, the Markov property and \eqref{fast}   give then 
$$
\mathbb{P}^o_\xi \big( S^o >t + \varepsilon\big)\,=\, \mathbb{E}^o_\xi \left[ \, \mathbf{ 1}_{ \{ S^o >t\} } \cdot \mathbb{P}^o_{X(t)} \big( S^o >  \varepsilon\big)\right]\,<\, \mathbb{P}^o_\xi \big( S^o >t    \big) 
$$
in conjunction with Proposition~\ref{L:Drift1}, and this establishes the strict decrease.
\end{proof}

\begin{proof}[Proof of Theorem~\ref{PP}]  We recall the reduction to natural scale in Subsection~\ref{SS scale}, as well as  the notation there. We also note that    
$$\int_{-\varepsilon}^{\varepsilon} \frac{\dx z}{ \, {\bm \sigma} ^4 (y+z)\,} \,=\, \infty\,, \qquad \forall \,\, \varepsilon \in (0, \min\{y-\lambda\,, \varrho-y\})
$$
holds for some fixed $y \in J$, if and only if 
$$
\int_{\mathfrak{q}(y-\varepsilon)-\mathfrak{q}(y)}^{\mathfrak{q}(y+\varepsilon)-\mathfrak{q}(y)} \,\frac{\dx \zeta}{  \mathfrak{s}^4(\mathfrak{q}(y)+\zeta)} = \int_{-\varepsilon}^{\varepsilon} \frac{\dx z}{ \mathfrak{p}^\prime\left(\mathfrak{q}(y+z)\right) \mathfrak{s}^4(\mathfrak{q}(y+z))} = \infty\,, \qquad \forall \,\, \varepsilon  \in \big(0, \min\{y-\lambda\,, \varrho-y\}\big)
$$
\newpage
\noindent
holds, since the derivative $\mathfrak{p}^\prime(\cdot)  = \exp(- 2 F(\cdot))$ of the scale function is continuous on $I$.  However, this last condition is equivalent to 
$$
\int_{-\delta}^{\delta}\, \frac{\dx z}{  \mathfrak{s}^4(\mathfrak{q} (y)+z)} = \infty\,, \qquad \forall \,\,\delta \in \big(0, \min\{\mathfrak{q} (y)-\ell, r-\mathfrak{q} (y)\}\big),
$$
so the conclusion follows now on the strength of Propositions~\ref{L:noDrift2} and \ref{P:noDrift}. 
\end{proof}
 
\begin{open}
 We have not been able to establish Theorem~\ref{PP} without the condition of Assumption~\ref{A1} --- or to find an example showing that it fails in the absence of this condition. We leave the resolution of this issue to future research. \qed
\end{open}

\section{Connections with  differential equations}   \label{S:DE}

We shall  now study  conditions, under which the function $U(\cdot, \cdot)$ and its Laplace transform can be characterized as the minimal nonnegative solutions of  appropriate linear partial and ordinary, respectively, differential equations. 

\subsection{Connections with   parabolic partial 
differential equations} \label{SS:pde}

In this subsection, we study conditions   implying that the function $U(\cdot, \cdot)$ 
 solves the Cauchy problem for the linear, parabolic partial differential equation
 \begin{equation}
\label{pdeU}
{ \partial \mathcal{U}  \over \partial \tau} (\tau, x) ={  \mathfrak{s}^2 (x)\over 2} {  \partial^2 \mathcal{U}  \over \partial x^2} (\tau, x)  + \mathfrak{ b}(x)  \mathfrak{s} ( x) {  \partial \mathcal{U}  \over \partial x} (\tau, x)
\end{equation}
with an appropriate initial condition, namely
 \begin{equation}
\label{icU}
\mathcal{U} (0 , x) =1,\qquad  x  \in   I.
\end{equation}

 We start with an existence result.
\begin{lem}{\bf Existence of a classical solution.}  \label{L classical}
	Assume that the functions $\mathfrak{s}(\cdot)$ and  $\mathfrak{b}(\cdot)$ are locally uniformly H\"older-continuous on $I$.
Then for any bounded, continuous function $g: (0, \infty) \times I \rightarrow \infty$ and any $n \in \N$, 
 the parabolic partial differential equation of \eqref{pdeU} has a unique classical solution $\,\mathcal{U}(\cdot\,, \cdot)$ of class ${\cal C}^{1,2}((1/n, \infty) \times (\ell_n, r_n))$, subject to the initial and lateral conditions
$$\mathcal{U}(1/n,x) = g(1/n,x) \qquad \text{and}  \qquad \mathcal{U}(\tau,\ell_n) = g(\tau,\ell_n),\quad   \mathcal{U}(\tau,r_n) = g(\tau,r_n)$$
for all $(\tau,x) \in (1/n, \infty) \times (\ell_n, r_n)$.
Moroever, the solution $\,\mathcal{U}(\cdot, \cdot)$ is bounded on $[1/n, \infty) \times [\ell_n, r_n]$.
\end{lem}

\begin{proof}
	The continuity of the function $\mathfrak{s}(\cdot)$ yields $\min_{x \in [\ell_n, r_n]} |\mathfrak{s}(x)| > 0$. Moreover, our assumptions imply that the functions $\mathfrak{s}^2(\cdot)$ and $\mathfrak{b}(\cdot)\mathfrak{s}(\cdot)$ are also uniformly H\"older-continuous on $[\ell_n, r_n]$. Thus, the existence and uniqueness result of Theorem~3.9 in \citet{Friedman_PDE}, and the maximum principle of Theorem~2.1 in this same book, yield the statement.
\end{proof}

We can now show that the function $U(\cdot, \cdot)$ of \eqref{U} solves the Cauchy problem of \eqref{pdeU}, \eqref{icU}.

\begin{prop}{\bf Stochastic representation of a solution to the Cauchy problem.}   
 \label{Prop3}
Under the assumptions of Lemma~\ref{L classical}, the function $U(\cdot\,, \cdot)$ is of class ${\cal C}([0, \infty) \times I) \cap {\cal C}^{1,2}((0, \infty) \times I)$ and solves the Cauchy problem of \eqref{pdeU}, \eqref{icU}.
\end{prop}

\begin{proof}
	We have shown the continuity of the function $U(\cdot \, , \cdot)$ in Proposition~\ref{P continuity}. We now fix $(T, \xi) \in (1/n,\infty) \times (\ell_n, r_n)$ for some $n \in\N$ and show that 
the function $U(\cdot\,, \cdot)$ satisfies the Cauchy problem of \eqref{pdeU} in $(1/n, \infty) \times (\ell_n, r_n)$, which then yields the statement. 
\newpage
 Applying Lemma~\ref{L classical} with $g(\cdot \,, \cdot) = U(\cdot \, ,\cdot)$, we see that this Cauchy problem has a bounded classical solution $\,\mathcal{U}(\cdot \,, \cdot)$. Simple stochastic calculus then implies that $\,\mathcal{U}(T-(t \wedge \rho), X(t \wedge \rho)), 0 \leq t \leq T$ is a bounded $\Prob_\xi$-martingale, 
where $\rho$ denotes the smaller of $T-1/n$ and $S_n$. Optional sampling gives
\begin{align*}
	\mathcal{U}(T, \xi) = \E_\xi\big[\,\mathcal{U}(T- \rho, X(\rho))\big] = \E_\xi\big[\,U(T- \rho, X( \rho))\big] = U(T,\xi),
\end{align*}
where the last equality is a consequence of the strong Markov property of the diffusion $X(\cdot)$.  We have shown that $U(\cdot \,, \cdot)$ coincides with $\,\mathcal{U}(\cdot \,, \cdot)$, and thus  solves the Cauchy problem of \eqref{pdeU}, \eqref{icU}.
\end{proof}

The proof of Proposition~\ref{Prop3} resembles the arguments in \citet{JT}.  For similar results, see Section~3.5 in \citet{McKean_1969}, Theorem~5.6.1 in \citet{Friedman_SDE},   \citet{HS}, \citet{BHS}, and \citet{Ruf_hedging}. We emphasize that the H\"older  exponent in Lemma~\ref{L classical} need not be $1/2$, as often postulated in related questions. Inspired by the observations made in Section~4.11 of \citet{ItoMcKean}, we expect that studying solutions in the sense of distributions would allow us to weaken the assumption  of H\"older continuity in Proposition~\ref{Prop3}; this will be the subject of future research. Considerable progress in this direction has been made by  \citet{Wang2014}, indeed in a more general, and multi-dimensional, setting.

This Cauchy problem of \eqref{pdeU}, \eqref{icU}   admits the trivial solution $ \,\mathcal{U}(\cdot\,, \cdot) \equiv 1$; it may have lots of other solutions. The one we are interested in, the function   $U(\cdot\,, \cdot)$ defined in \eqref{U}, turns out to be its minimal nonnegative solution. The following characterization of this function is  analogous to the results  in Problem~3.5.1 of \citet{McKean_1969}  and in~Exercise 4.4.7 of \citet{KS1}; see also \citet{FK, FK_uncert} and \citet{Ruf_hedging}.

Proposition~\ref{Prop3} yields directly the following corollary: 

\begin{cor}{\bf Continuity of density function.} 
Under the conditions of Lemma~\ref{L classical}, the explosion time $S$ has  a  continuous (sub-)probability density function.
\end{cor}

\begin{prop}{\bf Upper bounds on $\,U(\cdot\,, \cdot)$, and minimality.} 
 \label{Prop1}
The function  $U(\cdot \, , \cdot)$,  defined in \eqref{U} as the tail of the distribution function of the explosion time $S$,   is dominated by every nonnegative classical supersolution $\mathcal{U}(\cdot, \cdot)$ of the Cauchy problem of \eqref{pdeU}, \eqref{icU}. 

Furthermore, under the conditions of Lemma~\ref{L classical}, the function $U(\cdot\,, \cdot)$ is the smallest nonnegative classical (super)solution of \eqref{pdeU}, \eqref{icU}.
\end{prop}

\begin{proof}
Consider any continuous function $\, \mathcal{U} : [ 0 , \infty) \times I \ra [ 0 , \infty)$ of class $\,{\cal C}^{1,2} ( (0 , \infty) \times I )$ that satisfies the partial differential inequality
$$
{ \partial \mathcal{U} \over \partial \tau} (\tau, x) \ge{  \mathfrak{s}^2 (x) \over 2} { \partial^2 \mathcal{U}  \over \partial x^2} (\tau, x)  + \mathfrak{ b}(x) \mathfrak{s} ( x) {  \partial \mathcal{U}  \over \partial x} (\tau, x); \qquad (\tau, x) \in (0, \infty) \times I,
$$
as well as the initial condition $\,\mathcal{U} (0 , \cdot) \ge1$. 

For any given $T \in (0, \infty)$, it is checked readily on the strength of this inequality that the process  $\,\mathcal{U} ( T - (t \wedge S_n), X(t\wedge S_n) )$, $0 \le t \leq T$ is a local $\,\Prob_\xi-$supermartingale; as it is nonnegative, this process is actually a true $\,\Prob_\xi-$supermartingale, so 
\begin{equation*}
\mathcal{U} (T, \xi )  \ge  \E_\xi \big[ \, \mathcal{U} \big( T - (T \wedge S_n), X(T\wedge S_n) \big)  \big] \ge  \E_\xi \left[  \1_{\{S_n >T\}}\, \mathcal{U} ( 0, X(T ) )  \right]\ge \Prob_\xi \left( S_n >T \right) 
\end{equation*}
 by optional sampling. Letting $n \uparrow \infty$, we obtain the first claimed result from monotone convergence; the second then follows from Proposition \ref{Prop3}. 
\end{proof}

If there are no explosions, then the Cauchy problem of \eqref{pdeU}, \eqref{icU} actually has a unique bounded classical solution.

\begin{prop}{\bf Unique bounded solution.}   
\label{C:uniquePDE}
  Assume that the function $\,\mathcal{U}(\cdot \,, \cdot) \equiv 1$ is the smallest nonnegative classical solution of the Cauchy problem   \eqref{pdeU}, \eqref{icU}, and let   $V(\cdot \,, \cdot)$ be any bounded classical solution of this Cauchy problem. Then we have $V(\cdot \,, \cdot) \equiv \, \mathcal{U}(\cdot \,, \cdot) \equiv 1$.
\end{prop}
\begin{proof}
Let  $\,V: [0, \infty) \times I \ra [-K, K]$ be a bounded classical solution of the Cauchy problem   \eqref{pdeU}, \eqref{icU}, for some    $K\in (0,\infty).$  Then the function $\widehat{V}(\cdot, \cdot) = \big(V(\cdot, \cdot) + K\big)/(1+K)$ is a  bounded, nonnegative classical solution of the same Cauchy problem, so by assumption we must have   $\widehat{V}(\cdot, \cdot) \geq 1$.   

\smallskip
Let us assume that    $\widehat{V}(\cdot\,, \cdot)$ is not identically equal to the  constant function $\,\mathcal{U}(\cdot, \cdot) \equiv 1$, so that we have $K_2 := \sup_{(t,x) \in [0,\infty) \times (\ell, r)}$ $ \widehat{V}(t,x) \in (1,2)\,$. The function $\widetilde{V}(\cdot\,, \cdot) = \big(K_2- \widehat{V}(\cdot, \cdot)\big) /(K_2-1)$ is then a   classical solution of the Cauchy problem   \eqref{pdeU}, \eqref{icU} with values in $[0,1]$, and not identically equal to the  constant function $\,\mathcal{U}(\cdot\,, \cdot) \equiv 1$. But this  contradicts the assumption that the function $\,\mathcal{U}(\cdot\,, \cdot) \equiv 1$ is the smallest nonnegative classical solution of the Cauchy problem   \eqref{pdeU}, \eqref{icU}. 
\end{proof}

We note that it is not possible to remove the boundedness assumption in Proposition~\ref{C:uniquePDE}; see for example~\citet{Rosenbloom_1958}.  \citet{BX_2010} characterize one-dimensional time-homogeneous Cauchy problems with a unique solution.   For the relevance of super- and sub-solutions in the study of partial differential equations of parabolic type we refer to the recent paper by \citet{BS2012} and the references therein.  Let us also note that the characterizations of Propositions \ref{Prop1} and \ref{C:uniquePDE} are impervious to boundary conditions at the endpoints of the state space $\, I = (\ell, r)\,$.

\subsection{Connections with second-order ordinary differential equations}

Let us consider now, for any given real number $\la >0$, the {Laplace} transform or ``resolvent" of  the function $U( \cdot, \xi)$ in \eqref{U}, namely
\begin{align}
\label{hatU}
\widehat{U}_\la (\xi)=\int_0^\infty \exp(- \la T) U(T, \xi) \dx T = \int_0^\infty \exp( - \la T) \Prob_\xi(S>T) \dx T = { 1 \over \,\la\,} \Big( 1 - \E_\xi \big[ \exp( -\la S) \big] \Big). 
\end{align}

\begin{prop}{\bf Stochastic representation of a solution to an ordinary differential equation.} 
\label{Prop4}
If the functions $\mathfrak{b}(\cdot)$ and $\mathfrak{s}(\cdot)$ are continuous on $I$, then the function $\,\widehat{U}_\la (\cdot)\,$ is of class $\,{\cal C}^{2}(I)\,$ and satisfies the second-order ordinary differential equation
\begin{equation}
\label{u}
\frac{\mathfrak{s}^2 (x)}{2} \mathfrak{u}^{\prime \prime} (x) + \mathfrak{b}  (x) \mathfrak{s} (x)\mathfrak{u}^{\prime  } (x)  -  \la  \mathfrak{u}  (x) + 1 = 0, \quad x \in I.
\end{equation}
\end{prop}
\begin{proof}  
For some fixed $n \in \N$, we consider the ordinary differential equation
\begin{equation}
\label{u2}
{  \mathfrak{s}^2 (x) \over 2}v^{\prime \prime} (x) + \mathfrak{b}  (x)\mathfrak{s} (x) v^{\prime  } (x)  - \la  v (x)  = 0,\qquad x \in (\ell_n, r_n)
\end{equation}
with boundary condition $v(\ell_n) = v(r_n) = 0$, and note that it has the unique solution $v(\cdot) \equiv 0$. To see why, let us assume that \eqref{u2} has a non-constant solution $\widehat{v}(\cdot)$, and try to arrive at a contradiction.  This solution $\widehat{v}(\cdot)$ must have a local maximum or minimum at some $y \in (\ell_n, r_n)$ with $\widehat{v}^{\,\prime}(y) = 0$; assuming that $y$ is the location of a positive local maximum with $\widehat{v}(y) >0$, we obtain the absurdity $0>\mathfrak{s}^2 (y)\,\widehat{v}^{\,\prime \prime} (y)  = 2\lambda \widehat{v}(y)$. This yields the asserted uniqueness.

Thereom~12.3.1 in  \citet{Hartman1982} shows now that 
the differential equation in \eqref{u} has a unique solution in $(\ell_n, r_n)$ with boundary conditions $\mathfrak{u}(\ell_n) = g_1$ and  $\mathfrak{u}(r_n) = g_2$ for all $n \in \N$ and $g_1, g_2 \in \R$.

Finally, we fix a $\xi \in I$ and a sufficiently large $n \in \N$ so that $\xi \in (\ell_n, r_n)$ and let $\mathfrak{u}(\cdot)$ denote the solution of the differential equation in \eqref{u} with boundary conditions $\mathfrak{u}(\ell_n) = \widehat{U}_\la(\ell_n)$ and  $\mathfrak{u}(r_n) = \widehat{U}_\la(r_n)$. Simple stochastic calculus shows that the process
\begin{align}  \label{E M}
	M(t) := \exp \big(-\lambda(t \wedge S_n)\big) \left(\mathfrak{u}\big(X(t \wedge S_n)\big) - \frac{1}{\, \lambda \,}\right), \qquad  0 \leq t < \infty
\end{align}
 is a $\Prob_\xi$-local martingale; we conclude that $M(\cdot)$ is a uniformly integrable martingale, as it is bounded.  Since classical results, recalled in \eqref{eq:Sn finite}, yield   $\,\Prob_\xi(S_n<\infty) = 1$,  we obtain that 
 \begin{align*}
	\mathfrak{u}(\xi) - \frac{1}{\,\lambda\,} &= \E_\xi\left[\exp(-\lambda S_n) \left(\mathfrak{u}(X(S_n)) - \frac{1}{\,\lambda\,}\right)\right] = \E_\xi\left[\exp(-\lambda S_n) \left(\widehat{U}_\la(X(S_n)) - \frac{1}{\,\lambda\,}\right)\right] \\
	&= -\frac{1}{\,\lambda\,}\, \E_\xi\left[\exp(-\lambda S_n) \cdot \E_{X(S_n)}\left[\exp(-\lambda S)\right] \right]  =-\frac{1}{\,\lambda\,} \,\E_\xi\left[\exp(-\lambda S) \right] = \widehat{U}_\la(\xi) - \frac{1}{\,\lambda\,}\,;
\end{align*}
 the result now follows.
\end{proof} 

For related results, see Theorem~5.9.3 in \citet{Ito_2006}, and Theorem~13.16 on page~51 in Volume~II of \citet{Dynkin_1965}. 
Once again, the ordinary differential equation in \eqref{u} may have lots of classical solutions, in addition to the obvious  $ \mathfrak{u}  (\cdot) \equiv 1 / \la$. 

The function of \eqref{hatU} we are interested in, turns out to be the  smallest nonnegative classical supersolution of \eqref{u}.

\begin{prop}{\bf Upper bounds on $\,\widehat{U}_\la(  \cdot)$, and minimality.}
 \label{Prop5}
The function  $\widehat{U}_\la (  \cdot)$, defined in \eqref{hatU} as the Laplace transform of the tail of the distribution function of the explosion time $S$,   is dominated by every nonnegative classical supersolution of the second-order equation \eqref{u}.  

Furthermore, under the conditions of Proposition \ref{Prop4}, the function $\widehat{U}_\la (  \cdot)$ is the smallest nonnegative classical (super)solution of  \eqref{u}.
\end{prop}

\begin{proof} Consider any function $ \mathfrak{u}: I \ra [0, \infty)$ of class $ {\cal C}^2 (I)$ that satisfies 
$$
{ \mathfrak{s}^2 (x) \over 2} \mathfrak{u}^{\prime \prime} (x) + \mathfrak{b}  (x) \mathfrak{s} (x)\mathfrak{u}^{\prime  } (x)  -  \la   \mathfrak{u}  (x) + 1 \le  0, \quad x \in I.
$$
Simple stochastic calculus shows that the process $M(\cdot)$, defined in \eqref{E M}, is now a supermartingale, so
\begin{align*}
	\mathfrak{u}(\xi)    &
	\geq \E_\xi\left[\exp(-\lambda S_n)  \left(\mathfrak{u}(X(S_n)) - \frac{1}{\lambda}\right)\right] + \frac{1}{\lambda} 
	\geq \E_\xi \left[\int_0^\infty  \exp(-\la t)  \1_{ \{ S_n >t \} }  \dx t\right]     \\ 	&
	= \int_0^\infty \exp(-\la t) \Prob_\xi( S_n >t )  \dx t.
\end{align*}
Letting $n$ tend to infinity, we conclude that $$\, \mathfrak{u} (\xi)  \ge \int_0^\infty  \exp(-\la t) \Prob( S  >t ) \dx t =  \widehat{U}_\la (\xi)\,$$  holds for any given initial position $ \xi \in I$; the first claim follows. The second is   a consequence of Proposition \ref{Prop4}.
\end{proof}

By analogy with the situation in Subsection~\ref{SS:pde}, in the absence of explosions  the second-order equation  \eqref{u} actually has a unique bounded solution.

\begin{prop}{\bf Unique bounded solution.}   \label{C:uniqueODE}
  Assume that the function $\,\widehat{{\cal U}}_\la (  \cdot) \equiv 1/\la \,$ is the minimal nonnegative classical solution of the second-order equation \eqref{u}, and let   $\,\mathfrak{u}(\cdot)\,$ be any bounded classical solution of   \eqref{u}. Then we have $\,\mathfrak{u}(\cdot) \equiv \,\widehat{{\cal U}}_\la (\cdot)\,$.  
 \end{prop}

\begin{proof}
This result can be proved in exactly the same manner as Proposition~\ref{C:uniquePDE}.
\end{proof}

Again, it is not possible to remove the boundedness assumption in Proposition~\ref{C:uniqueODE}. For instance, with $I=\R$, $\mathfrak{ s}(\cdot) \equiv \sqrt{2\,}$, $\mathfrak{b}(\cdot) \equiv 0$, and $\lambda =1$ in \eqref{u}, the (smallest nonnegative and) unique bounded solution of the second-order equation \eqref{u} is of course $\widehat{\mathcal{U}}_\la (\cdot) \equiv 1$,  but there is a host of unbounded solutions $\mathfrak{u}_{\kappa, \vartheta} (x) = 1 +  \kappa \exp(x)   +  \vartheta \exp(-x)$ for all $\kappa, \nu \in \R$ (nonnegative, as long as $\kappa \ge 0\,,~ \vartheta \ge 0$).

\smallskip
Finally, we remark that solving the ordinary differential equation in \eqref{u2} has been the standard way to compute Laplace transformations of the distributions of hitting times, a computation that  is a special case of the computation of the distributions of explosion times; see for example \citet{Barndorff_1978}, \citet{Kent_1978}, or \citet{Salminen_Vallois}.

\subsection{Equivalent formulations for the Feller test}  \label{SS:equivalent}

The next theorem summarizes several of our previous observations; it amounts to an extended version of the celebrated Feller test for explosions, which we recall in Appendix~\ref{FT}.  
\begin{thm}{\bf Characterization of explosions.} 
\label{Thm2}
The following conditions are equivalent:
\begin{itemize}
	\item[(i)] the diffusion process $X(\cdot)$ of \eqref{1} has no explosions, i.e., $\,\Prob (S = \infty) = 1\,; $
	\item[(ii)] $w(\ell +) = w(r-) = \infty\,$ hold  for the ``Feller test" function defined in \eqref{E v}
for some $c \in (\ell, r)$.
\end{itemize}

If the function $\mathfrak{f}(\cdot)$ of \eqref{f} is locally square-integrable on $I$, then conditions (i)-(ii) are equivalent to:
\begin{itemize}
	\item[(iii)] the  truncated exponential $\,\Prob^o-$supermartingale
\begin{equation*}
Z^\flat(T)= \exp \left( \int_0^T \mathfrak{b} ( X^o (t) )  \dx W^o (t) - {1 \over \,2\,} \int_0^T \mathfrak{b}^2 ( X^o (t) )  \dx t  \right) \cdot \1_{ \{ S^o > T\} }, \qquad 0 \le T < \infty,
\end{equation*}
appearing in \eqref{16},
is a $\,\Prob^o$-martingale.
\end{itemize}
If the functions $\mathfrak{s}(\cdot)$ and  $\mathfrak{b}(\cdot)$ are continuous on $I$, conditions (i)--(iii) are equivalent to:
\begin{itemize}
	\item[(iv)]  the smallest nonnegative classical solution of the second-order differential equation in  \eqref{u}  is		$\mathfrak{u}( \cdot) \equiv 1/\la\,;$
	\item[(iv)$^\prime$]  the unique   bounded classical solution of the   equation in  \eqref{u}  is		$\,\mathfrak{u}( \cdot) \equiv 1/\la\,$.
\end{itemize}
If, in addition, the functions $\mathfrak{s}(\cdot)$ and  $\mathfrak{b}(\cdot)$ are locally uniformly {H\"older}-continuous on $I$, conditions (i)--(iv)  are equivalent to:
\begin{itemize}
	\item[(v)] the smallest nonnegative classical solution of the Cauchy problem  \eqref{pdeU}, \eqref{icU} is $ \,\mathcal{U}( \cdot, \cdot) \equiv 1\,;$ 
	\item[(v)$^\prime$] the unique   bounded classical solution of the Cauchy problem  \eqref{pdeU}, \eqref{icU} is $ \,\mathcal{U}( \cdot, \cdot) \equiv 1$.
	\end{itemize}
\end{thm}

\begin{proof}
The equivalence of (i) and (ii) is the subject of the {Feller}  test for explosions \citep[see for example Theorem~5.5.29 in][]{KS1}. The equivalence of (i) and  (iii) follows  from Theorem~\ref{Thm1}.
Under the stated conditions, the  equivalence of (i) and (iv) is covered by Propositions~\ref{Prop4} and \ref{Prop5}; whereas the equivalence of (i) and (v)   is covered by Propositions~\ref{Prop3} and \ref{Prop1}.   The equivalence of (iv) and (iv)$^\prime$ (resp., (v) and (v)$^\prime$) is the subject of Proposition~\ref{C:uniqueODE} (resp., Proposition~\ref{C:uniquePDE}).
\end{proof}

It would be of some  interest  to have a more ``circular" proof of this result, in particular, a direct derivation of the {Feller} test (ii) from one of the minimality  properties (iv), (v).

\section{Examples}\label{S:examples}

 Let us consider some illustrative examples. In several of these examples, we will rely on the Lamperti transformation, reviewed in Appendix~\ref{Doss}.

\begin{exmp}{\bf Reciprocal of Brownian motion.} 
 \label{Ex1} 
Let us take $I = (0, \infty)$ and 
\begin{equation}
\label{20}
\mathfrak{s} (x) = -x^2, \qquad  \mathfrak{ b}(x)= -x,
\end{equation}
implying $\,\mathfrak{ f} (x) = 1 / x$. With a given initial condition $ \xi \in (0, \infty)$, the driftless equation of \eqref{13} becomes  
$$
  X^o (\cdot)= \xi - \int_0^{ \cdot} \big( X^o (t)\big)^2\,  \dx W^o (t),
$$
and is easily seen to take values in $ I = (0, \infty)$ for all times, as it is identified with the reciprocal $X^o (\cdot) = 1 / R(\cdot)$ of the three-dimensional {Bessel} process 
\begin{equation*}
\dx R (t)=  { 1 \over  R(t)} \dx t + \dx W^o(t), \qquad R(0) = { 1 \over \, \xi \,}\,.
\end{equation*}
In particular, the condition in \eqref{DS0} is satisfied, and we have $\Prob^o (S^o = \infty) =1$. It has been known since the work of \citet{JohnsonHelms} that $X^o(\cdot)$ is a strict local martingale; and   the connection with the {Bessel} process allows the computation of the distribution of $ X^o (T)$ as in \citet{ET}, namely
\begin{equation*}
\Prob ( X^o (T) \in \dx y)  ={ \xi  \over   y^{3} \sqrt{ 2   \pi   T} }\left( \exp \left(- \frac{( (1/y) - (1/\xi) )^2}{2T} \right) - \exp \left( - \frac{( (1/y) + (1/\xi) )^2}{2T} \right)\right).
\end{equation*}

On the other hand, with the choices of \eqref{20}, the   equation of \eqref{1} becomes
\begin{equation}
\label{19}
 X  (\cdot) =\xi - \int_0^{ \cdot}  X^2  (t) \dx W (t) + \int_0^{ \cdot} X^3  (t) \dx t\,,
\end{equation}
whereas the functions of Subsection~\ref{SS feynman} take the form $F(x) = \log (x)$  and $V(\cdot) \equiv 0$. The condition  in \eqref{DS} of the Feller test  clearly fails in this case, so we have $\, \Prob (S= \infty) <1\,$; in fact, it follows   from \eqref{17} that 
\begin{equation}
\label{19a}
\Prob ( S>T)\,=\, { 1 \over  \,\xi \,} \cdot \E^o \left[{ 1 \over  R(T)} \right]\,=\,2 \int_0^{1 / ( \xi \sqrt{T})}  { 1 \over    \sqrt{2  \pi}} \exp\left( - { r^2 \over 2} \right)  \dx r\,=:\,U (T, \xi) 
\end{equation}
holds for $\,0 < T < \infty\,$, so in fact $\, \Prob (S= \infty) =0\,$. This is also quite straightforward to check directly, as follows: we  observe that \eqref{19} is of the special   
form 
\begin{equation}
\label{22}
\dx X(t) = \mathfrak{s} ( X(t) ) \left(  \dx W (t) + { 1 \over \,2\,} \, \mathfrak{s}' ( X(t) )\, \dx t\right), \qquad X(0) = \xi\,,
\end{equation}
or equivalently the function $\mathfrak b(\cdot)$ is of the form \eqref{30} with $\mub  = 0$.   Following the procedure of Appendix~\ref{Doss}, we see that \eqref{19}  can be solved ``pathwise" in terms of the function $\vartheta_\xi (w ) = ( w + (1 / \xi) )^{-1}$, namely as
$$
X(t) ={ 1 \over W(t)+  (1 / \xi) }\,, \qquad 0 \le t < S\,.
$$
Thus, the explosion time $S$ is the first hitting time of the level $ (-1 / \xi)$ by a standard Brownian motion started at the origin, that is, the right-hand side of \eqref{19a}. Note that the process explodes at $S = \lim_{n \uparrow \infty} S_n'$ where $S_n' = \inf \{ t \ge 0: X(t) \ge n\}$; the left endpoint $\ell =0$ of the state space is inaccessible by $X(\cdot)$. This is consistent with the observation made in Proposition~\ref{P conditions}(i) since $\int_0^1 \mathfrak{s}^{-1}(z) \dx z = \infty$ and $\int_1^\infty \mathfrak{s}^{-1}(z) \dx z = 1 < \infty$.

Here, it is easy to verify ``by hand" that the function $ U (\cdot \,, \cdot)$,  defined in \eqref{19a}, satisfies both the linear parabolic equation of \eqref{pdeU}, now in the form
$$
{ \partial \mathcal{U}  \over \partial \tau} (\tau, x) ={  x^4\over 2} {  \partial^2 \mathcal{U}  \over \partial x^2} (\tau, x)  +   x^3 {  \partial \mathcal{U}  \over \partial x} (\tau, x), \qquad (\tau, x) \in (0, \infty) \times I,
$$
and the initial condition $\mathcal{U} (0+, x) \equiv 1$ for all $x \in I$. (It also satisfies  the lateral condition $ \mathcal{U} ( T, 0+) \equiv 1$ for all $T\in (0, \infty) $ but this is immaterial, as the left endpoint $ \ell =0$ of the state space is inaccessible.) The function of \eqref{19a} is not the only classical solution of this Cauchy   problem, as $ \,\mathcal{U} (\cdot, \cdot) \equiv 1\,$ is clearly a solution; from Proposition~\ref{Prop1}, however, $U(\cdot, \cdot)$ is its {\it smallest nonnegative classical solution. }

 Let us consider next the second-order ordinary differential equation in \eqref{u},   written here as
\begin{align} \label{eq: uInEx1}
	\frac{x^4}{2} \mathfrak{u}^{\prime \prime} (x) + x^3 \mathfrak{u}^{\prime  } (x)  -  \la  \mathfrak{u}  (x) + 1 = 0\,, \qquad x \in I.
\end{align}
It is easy to see that a general solution of this differential equation takes the form
\begin{align*}
	\mathfrak{u}(x) = A \exp\left(- \frac{2 \lambda}{x}\right) +B \exp\left(\frac{2 \lambda}{x}\right) + \frac{1}{\lambda}\,, \qquad x \in I
\end{align*}
for some real constants $A,B  $. We are interested in the smallest nonnegative solution $\widehat{U}_\lambda(\cdot)$ in \eqref{eq: uInEx1}, which we obtain by first setting $B=0$ and then $A = -1/\lambda$. We thus have $$\widehat{U}_\lambda(\cdot) = \frac{1}{\, \lambda \,} \left(1 - \exp\left(- \frac{2 \lambda}{x}\right)\right),$$
which is clearly smaller than the constant   $1/\lambda$.  This illustrates the validity of the characterization  of an explosive diffusion in (v) of Theorem~\ref{Thm2}.
\qed
\end{exmp}

 \begin{exmp} \label{ex:Bessel}  {\bf Stochastic equations of Bessel-type.}  With a given constant   $ \delta \in (-\infty, 2)\,$ and state space $I = (0, \infty)$, let us consider the stochastic equation 
\begin{equation}
\label{BessEq}
 \dx X(t) = { \delta - 1 \over 2 X(t)} \dx t + \dx W(t),\qquad X(0) = \xi\,.
\end{equation}
The solution of this equation does not explode to infinity, but reaches the origin in finite time: $ \Prob (S < \infty) = 1$, where $S = \lim_{n \uparrow \infty} S_n'$ with $S_n' = \inf\{ t \ge 0: X(t) \le 1 / n\}$. When $\, \delta \in (1,2)\,$, this corresponds to a   Bessel  process, with ``dimension parameter" $ \,\delta \, $  and absorption at the origin. 

In our notation $\mathfrak{s} (\cdot) \equiv 1$ and, with index $\,\nu = 1 - (\delta/2) > 0$, we have   
$$
 	\mathfrak{f}(x) =  \frac{1/2-\nu}{x}\,, \qquad F(x) = \log \left( x^{  1/2-\nu }\right),\qquad V(x) = { \nu^2 - 1/4 \over 2 x^2}
$$ 
in the notation of Subsection~\ref{SS feynman}.
The representation in \eqref{17} then helps us compute the distribution of $S$ as the expectation of a  functional of the Brownian motion $X^o (t) = \xi + W(t)$, $0 \le t < S^o$ with $ S^o = \inf\{ t \ge 0: \xi + W(t) =0\}$; to wit, 
\begin{align}
\Prob ( S>T) &= \E^o \left[\left( \frac{X^o (T)}{\xi}\right)^{-2\nu} \cdot \left(\frac{X^o (T)}{\xi}\right)^{\nu+1/2}  \exp \left(   \frac{\,1/4-\nu^2\,}{2} \int_0^T { \dx t \over ( X^o (t) )^2} \right) \cdot \1_{ \{ S^o >T\}}   \right]         \nonumber   \\
	&= \E^{\Q^\nu} \left[\left( \frac{X^o (T)}{\xi}\right)^{-2\nu} \right]        \label{30before}    \\
&=  \frac{1}{\,T\,}\,  \xi^\nu \exp\left(\frac{-\xi^2}{2T} \right) \int_0^\infty
x^{1-\nu}  \exp\left(\frac{ - x^2}{2T}\right) I_\nu \left( \frac{\xi x}{ T} \right) \dx x \,.  \label{30a}
\end{align}
Here,  $\Q^\nu$ is the probability measure under which the process $\,X^o(\cdot)\,$ is   Bessel   of dimension $\,2\nu+2 = 4-\delta>2\,$; the change of measure is proved in Exercise~XI.1.22, and the density of the process $X^o(\cdot)$ is  derived on page~446, of \citet{RY}.   
 In \eqref{30a} and  throughout the paper, we denote by $\,I_\nu(\cdot)\,$ the modified {Bessel} function of the second type, namely 
\begin{equation}
\label{ModBes}
I_\nu (u)  \,  :=  \sum_{n \in \N_0} \frac{ ( u /2)^{\nu + 2 n}}{\, n! \, \Gamma ( n + \nu +1)\,} .
\end{equation}

Now with the help of the monotone convergence theorem and of the substitutions $z = x/\sqrt{2T}$ and $y = \xi^2/(2T)$, the expression in \eqref{30a} simplifies   to 
\begin{align*}
\Prob ( S>T) &= 2  y^{\nu/2} \exp\left(-y \right) \int_0^\infty
z^{1-\nu}  \exp\left( - z^2\right) I_\nu \left( 2 z \sqrt{y} \right) \dx z \\
&= 2  y^{\nu/2} \exp\left(-y \right) \sum_{k = 0}^\infty\, \frac{y^{\nu/2+k}}{\,k! \, \Gamma(\nu+k+1)\,} \int_0^\infty
z^{1+2k}  \exp\left( - z^2\right) \dx z   \\
&=\exp\left(-y \right) \sum_{k = 0}^\infty \frac{y^{ \nu+k}}{\Gamma(\nu+k+1)}  \\
&= \frac{1}{\Gamma(\nu)} \int_0^{y} \exp  \left(-\theta \right) \theta^{\nu-1} \dx \theta =  \frac{1}{\Gamma(\nu)} \int_0^{\, \xi^2/(2T)} \exp  \left(-\theta \right) \theta^{\nu-1} \dx \theta   \,.
\end{align*}
Thus, we have
\begin{align} 
\label{GH}
\Prob ( S>T)=\Prob \left( \mathfrak{G} < { \xi^2 \over  2  T } \right) =H_\nu \left(   { \xi^2 \over 2  T } \right),\quad \hbox{where} \quad H_\nu (u) := \int_0^u \,{  t^{\nu-1} \exp(-t) \over \Gamma (\nu)}  \dx t
\end{align}
is the cumulative Gamma($\nu$) probability distribution function, and the random variable $\mathfrak{G}$ has Gamma distribution with parameter $\, \nu\,$.

Of course, it is well known from the time-reversal considerations in Section~2.1 of \citet{Going_Yor} -- based on the results in \citet{Getoor_1979} and \citet{Pitman_Yor_1981} -- that $S$ has the distribution of $\xi^2 / (2 \mathfrak{G})$; see also Section~14 in \citet{Kent_1982}. Here we just derived this fact from  rather elementary Bessel process computations, with no need for time-reversal.  Alternatively, Proposition~6.1 in \citet{Yor_2001} yields the representation of the explosion time $S$ as in \eqref{GH} by using the identity in \eqref{30before} but now via a representation  of negative powers as an integral of exponentials.

 \smallskip
 It can be checked ``by hand" that the  function $ (T, \xi) \mapsto \Prob_{\xi} ( S>T)  =H_\nu \left(     \xi^2 /( 2  T) \right)$ in \eqref{GH} is a classical solution of  the  linear parabolic equation
 $$
 {  \partial \mathcal{U}  \over  \partial T} (T, \xi) ={ 1\over \,2\,} { \partial^2 \mathcal{U}  \over \partial \xi^2} (T, \xi)  +   {\, \delta - 1\, \over 2 \,\xi\,} \, {  \partial\mathcal{U}  \over \partial \xi} (T, \xi), \qquad (T, \xi) \in (0, \infty) \times I,
 $$
 and satisfies also the initial condition $ \,\mathcal{U} (0+, \xi) = 1$ for  $\xi\in I$  $($as well as   the lateral condition $\, \mathcal{U} (T, 0+ ) = 0$ for  $T\in (0, \infty))$.  From Proposition \ref{Prop1}, this function is the {\it smallest nonnegative classical solution} of the initial / boundary value problem under consideration. \qed
   \end{exmp}

\begin{rem}{\it A generalization of Example~\ref{ex:Bessel}.} 
	\citet{Barndorff_1978} construct, for a given set of parameters, a one-dimensional diffusion $X(\cdot)$ on the interval $I=(0,\infty)$ such that the corresponding time to explosion $S$ has   a generalized Gamma distribution.  If the parameters are chosen so that the generalized Gamma distribution is exactly a Gamma distribution, then their construction yields the Bessel process of  	Example~\ref{ex:Bessel}.
\qed
\end{rem}

\begin{exmp} {\bf   A generalization of Example~\ref{Ex1}, and $\mathbf{ h}$-transforms.}  \label{Ex h-transform}
The special case $I = (0, \infty)$ and $\,\mathfrak{f} (x) = 1 / x$, or equivalently $\mathfrak{ b} (x) = \mathfrak{s} (x) / x$, corresponds exactly to the situation in which the diffusion process $X(\cdot)$ of (\ref{1}) is the $h$-transform of the nonnegative local martingale and diffusion  in natural scale $X^o (\cdot)$ of \eqref{13};   see for example \citet{Doob_1957} or Section~3.3 in \citet{Perkowski_Ruf}.  We observe that Example~\ref{Ex1} is a special case of this setup. The above choices lead to $ F(x) = \log (x)$  and $V(\cdot) \equiv 0$ in the notation of Subsection~\ref{SS feynman}, and  \eqref{17} becomes 
\begin{equation}
\label{17a}
\Prob ( S>T )=  { 1 \over \,\xi\,} \,\E^o\left[ X^o (T) \cdot \1_{ \{ S^o >T\} }\right]  = { 1 \over \,\xi \,}\, \E ^o\big[\,X^o (T  \wedge S^o ) \,\big], \qquad 0 \le T < \infty \,.
\end{equation}
This can be seen from first principles, as $X^o(\cdot)$ is a nonnegative local martingale, which gets absorbed at the origin the first time it reaches it. 
Consequently, we have $\Prob (S = \infty) =1$ if and only if the stopped diffusion in natural scale  $X^o (\cdot)$ of \eqref{13} is a   martingale. 

Let us recall from  Proposition~\ref{P conditions}(iii) that we have $w(\infty ) = \infty$ and $\,\Prob(S=\infty) = 1\,$, if and only if the condition in \eqref{DS} holds.
That is, $X^o (\cdot)$ is a martingale, if and only if the condition in  \eqref{DS} holds; see also \citet{DShir}.
Alternatively, we observe that the process $Y(\cdot) = 1/X(\cdot)$ satisfies the equation
\begin{equation*}
 Y  (\cdot) = \frac{1}{\,\xi\,} - \int_0^{ \cdot}    Y^2(t) \cdot  \mathfrak{s}\left(\frac{1}{Y(t)}\right)  \dx W (t)
\end{equation*}
and reaches the origin if and only if 
\begin{equation*}
\int_{0}^1  {  z \over \,z^4 \, \mathfrak{s}^2 (1/z) \,} \,\dx z = \infty\,, 
\end{equation*}
by virtue of \eqref{DS0}. This, however, is again equivalent to \eqref{DS}.

The equation for $Y(\cdot)$ can be solved by the familiar method of time-changing a standard Brownian motion $ B(\cdot)$, namely $Y  (t) ={1/\xi} +  B( A ( t) )$; here 
$$
A ( t) = \int_0^{ t} h^2 ( Y (s) )\, \dx s\,, \qquad  0 \le t < \infty\,,
$$
with $\,h(y) := y^2 \,\mathfrak{s}(1/y)\,$, is the inverse of the continuous, strictly increasing, real-valued  function 
$$
\Gamma (u) :=\int_0^u  h^{-2} \left(    {1 \over \,\xi \,}  +  B ( v) \right)  \dx v, \qquad 0 \le u < \infty 
$$
with $ \Gamma(\infty)=\infty$. The explosion time $S$ of $X(\cdot)$ is thus related to the first hitting time 
$$
\taub :=\inf \left\{ u \ge 0: B(u) =-{1 \over \, \xi \,}\right\}
$$
for the auxiliary Brownian motion $B(\cdot)$ via
\begin{equation}
\label{28}
\Prob ( S>T)  =  \Prob \big(\taub > A (T)\big) = \Prob \big( \Gamma( \taub ) >T\big).
\end{equation}

We now consider the special case $\mathfrak{s}(x) = \kappa x^p$ for all $x > 0$ for some real numbers  $\,\kappa>0\,, \,p >0\,$.  First, by \eqref{DS0}, we observe that $X^o(\cdot)$ is a martingale and $\Prob(S=\infty) =1$, if and only if $p\leq 1$, that is, the function $\mathfrak{s}(\cdot)$ grows at most linearly.  We note that in the case $\,p = 1/2\,$ we have 
 \begin{equation}  
 \label{24}
 \dx X  (t) = \kappa^2 \dx t + \kappa \sqrt{X  (t)\,}   \,  \dx W(t), \qquad X  (0)= \xi. 
  \end{equation}  
  In particular, with $\kappa = 2$ we see that $X(\cdot)$ is the square of a {Bessel} process in dimension $\delta = 4$. 

For the  case $\,p>1=\kappa\,$ we would like to compute the distribution of the random variable 
\begin{equation}
\label{GXM}
\Gamma(\taub) =\int_0^\taub \big( x + B(t) \big)^{\mu}  \dx t  \,, \quad \text{where we have set}~~~ x = { 1 \over \,\xi\,}\,, ~~~ \mu = 2p-4\,,
  \end{equation} 
and thus the distribution of the explosion time $S$ via \eqref{28}. Let us consider two special cases first.

(i)~ For the value $\,p = 2$, we have $h(\cdot) \equiv 1$, $ \Gamma(u) \equiv u$ and $ \Prob ( S>T)  = \Prob (  \taub   >T)$ is given by   \eqref{19a}. 

(ii) For the value $\,p = 3/2$,   \citet{Borodin_handbook}  provide in Formula~(2.19.2) on page~208 the distribution of the random variable 
$$
\Gamma(\taub) = \int_0^\taub { \dx v \over  \xi^{-1} + B(v)} \qquad \mathrm{as} \qquad \Prob \big( \Gamma (\taub ) \in \dx t \big)= \frac{2}{\,\xi  t^{2}\,} \exp \,\left(-  { 2 \over \, \xi t \,} \right)\dx t
$$ 
or equivalently the distribution of the explosion time 
\begin{equation}
\label{28a}
 \Prob ( S>T) = \Prob ( \Gamma( \taub ) >T)=1 - \exp \left(-  { 2 \over  \xi  T} \right)= { 1 \over  \,\xi\,} \cdot \E^o \left[ X^o(T)  \right], \qquad 0 < T < \infty 
\end{equation}
for the diffusion
\begin{align*}
 X  (\cdot) &= \xi + \int_0^{ \cdot}  \big( X(t) \big)^{3/2} \dx W (t) + \int_0^{ \cdot}  \big(X  (t) \big)^2 \dx t\,,
\end{align*}
thus
$\,  X^o  (\cdot)  = \xi + \int_0^{ \cdot}  (X^o(t))^{3/2} \dx W^o (t)\,$; 
   a related diffusion is discussed in Example~\ref{3/2}. It is   checked ``by hand", that the function $U (T, \xi) = 1 - \exp ( -2 / (\xi T ) )$ in \eqref{28a} satisfies  the linear parabolic equation 
\begin{equation}
\label{LPE}
{  \partial \mathcal{U} \over \partial T} (T, \xi)\,=   \, \mathfrak{s}^2 (\xi)   \left( {  \, 1\, \over 2}{ \partial^2 \mathcal{U} \over \partial \xi^2} (T, \xi)  +    \frac{1}{\,\xi\,} \,{ \partial \mathcal{U}  \over \partial \xi} (T, \xi) \right)\,, \qquad (T, \xi) \in (0, \infty) \times I
\end{equation}
subject to the initial condition $\,\mathcal{U} (0+, \cdot) \equiv 1$ on $  I$, for $\,\mathfrak{s}  (\xi) = \xi^{\,3/2}\,$; from Propositions~\ref{Prop3} and \ref{Prop1}, the function $\,U\,$ of \eqref{28a} is the smallest nonnegative (super)solution of the Cauchy problem.

 In order to compute the distribution of the random variable $\, \Gamma (\taub)\,$ of \eqref{GXM} in any generality, the crucial observation, for which we are grateful to   Marc Yor, is that for any given $x>0$, $\varrho>0$ and for standard Brownian motion $B(\cdot)$,   the representation
\begin{equation}
\label{LJY}
\big( x + B(t) \big)^{2\varrho}\,=\, R \left(k \int_0^t \big( x + B(s) \big)^{\mu}\,\dx s \right), \qquad 0 \le t \leq \taub
\end{equation}
holds, where $R(\cdot)$ is a Bessel process started in $x^{2 \varrho}$ with dimension
\begin{equation}
\label{DCM}
\delta\,=\, 2 - \frac{1}{\,2\varrho\,}\,, \qquad \text{and}~~~k=4 \varrho^2\,, ~~~\mu = 2 \big( 2 \varrho -1 \big)\,, ~~~x=1/\xi\,.
\end{equation}
This is verified easily, via stochastic calculus and the dynamics in \eqref{BessEq} for the process $R(\cdot)$; consult also Proposition~XI.1.11 in \citet{RY}. However, now the representation \eqref{LJY} identifies $\,k \,\Gamma (\taub)\,$ as the first time the Bessel process $R(\cdot)$ visits the origin, and \eqref{28}, \eqref{GXM}, \eqref{DCM} give the distribution of the explosion time $S$ as 
$$
P \big( S>T) \,=\, H_\nu \left( \frac{x^{\,4\varrho}}{\,2\, k\,T\,}\right)\,=\, H_\nu \left( \frac{2\, \nu^2}{\, T \, \xi^{\,1/\nu}\,}\right), \quad \text{with}~~~\nu = 1 - \frac{\delta}{\, 2 \,} = \frac{1}{\, 4 \varrho\,}\quad \text{and}~~~\varrho = \frac{p -1}{2}>0
$$
in the notation of \eqref{GH}; see also \citet{Hammersley}, \citet{Getoor_Sharpe},  \citet{Pitman_Yor_1981}, and Proposition~1 in \citet{Khos_Salminen_Yor}.  Once again, it is checked by hand that this function solves the equation \eqref{LPE} subject to the initial condition $\,\mathcal{U} (0+, \cdot) \equiv 1$ on $  I$, for $\,\mathfrak{s}  (\xi) = \xi^{\,p}\,$, $\,p>1$; and from Propositions~\ref{Prop3}, \ref{Prop1}  that it is the smallest nonnegative (super)solution of this Cauchy problem. 
\qed
 \end{exmp}

In  Example~\ref{Ex h-transform} (cf.~\eqref{17a}), we proved the following result, which appeared under slightly stronger assumptions in \citet{DShir}.


\begin{cor}{\bf Martingale property of nonnegative diffusions in natural scale.}   
\label{C:mg property}
Suppose that    the function $1/\mathfrak{s}(\cdot)$ is locally square-integrable on $I = (0,\infty)\,$, and that $X^o(\cdot)$ is a nonnegative $\Prob^o-$local martingale  which  satisfies  the stochastic differential equation $$X^o(\cdot) \,=\, \xi + \int_0^\cdot \mathfrak{s}\big(X^o(t)\big)\, \dx W^o(t) $$ of  \eqref{13}, and becomes absorbed at the origin the first time it gets  there. 
Then  $X^o(\cdot)$ is a true $\Prob^o-$martingale,  if and only if the condition of  \eqref{DS} holds, namely $$ \int_{1}^\infty  \frac{\,z\,\dx z\,}{\, \mathfrak{s}^2 (z) \, } \, =\,\infty\,.$$ 
\end{cor}

\begin{exmp}{\bf Lamperti transformation for affine variance diffusions.}
 Let us consider a modification of the stochastic differential equation \eqref{24}, namely $I = (0, \infty)$ and
\begin{equation}
\label{25}
 \dx X  (t) ={ \kappa^2  \over 4} \,\dx t + \kappa \sqrt{X  (t)} \,\dx W(t), \qquad X  (0)= \xi   \in (0, \infty)
 \end{equation}
for some $\kappa>0$.
This corresponds to  $\,\mathfrak{s} (x) = \kappa \sqrt{ x\,}\,$, $\,\mathfrak{ b}(x)= \kappa / (4\sqrt{ x\,}\,)$, thus 
$$ 
\mathfrak{ f} (x) = {1   \over  4 x} , \qquad  F(x) = {1   \over \, 4 \, } \log(x),   \qquad \hbox{ and } \qquad V(x)= -{3  \kappa^2 \over \, 32  x\,}\,.
$$ 

We note that $\mathfrak{ b}(x) = \mathfrak{s}'(x) / 2$ holds,   and thus  Appendix~\ref{Doss} applies. 
The result is $ X(\cdot) = \vartheta_\xi ( W(\cdot) )$, now in terms of the function  
$$
 \vartheta_\xi (w ) = \left( \frac{\kappa}{\,2\,} \,w +   \sqrt{\xi\,} \,\right)^{2} , \quad \mathrm{namely} \quad X(t) =\left( {\kappa  \over \,2\,}\, W(t) + \sqrt{\xi\,} \,\right)^2 , \quad 0 \le t < S.
$$
It follows that $S$ is the first hitting time of the point $( -2 \sqrt{ \xi \,}/\kappa)$ by standard Brownian motion started at the origin. Thus, the process $ X(\cdot)$ explodes by hitting the left endpoint $0$ of the state space; the right endpoint $\infty$ is inaccessible by the diffusion $X(\cdot)$; that is, $X(S-) = 0$.  Of course, this is again consistent with Proposition~\ref{P conditions}(i) since $\int_0^1 \mathfrak{s}^{-1}(z) \dx z = 2 \kappa<\infty$ and $\int_1^\infty \mathfrak{s}^{-1}(z) \dx z =  \infty$.

From \eqref{17}, we conclude that
\begin{align}
\label{29a}
\Prob ( S>T) &= 
2 \int_0^{\,2 \sqrt{\xi}/ (\kappa \sqrt{T})}  { 1 \over \sqrt{2  \pi}} \exp\left( - { y^2 \over 2} \right)  \dx y\\  &= \E^o \left[\left(\frac{ X^o (T)}{\xi} \right)^{1/4} \exp \left(   {  3  \kappa^2\over 32} \int_0^T { \dx t \over X^o (t)} 
\right) \cdot \1_{ \{ S^o >T\} } \right]; \nonumber
\end{align}
 here $X^o (\cdot)$ is a martingale  and diffusion in natural scale  with
  \begin{equation*}
 \dx X^o  (t) = \kappa \sqrt{X^o  (t)} \, \dx W^o(t), \qquad X^o  (0)= \xi. 
  \end{equation*}  
   
 Once again, it is easy to check  by direct computation that the function of \eqref{29a} solves the 
 linear parabolic equation of \eqref{pdeU}, namely 
 $$
 { \partial \mathcal{U} \over \partial \tau} (\tau, x) ={ \kappa^2 x \over 2} { \partial^2 \mathcal{U}  \over \partial x^2} (\tau, x)  +   {\kappa^2 \over 4}  {  \partial \mathcal{U} \over \partial x} (\tau, x), \qquad (\tau, x) \in (0, \infty) \times I,
 $$
 subject to the initial condition $ \mathcal{U} (0+, x) = 1$ for  $x\in I$  (and to the lateral condition $ \mathcal{U} (T, 0+ ) = 0$ for  $T\in (0, \infty))$. Arguing then by analogy with Proposition~\ref{Prop1}, it is checked that the function of \eqref{29a} is the {\it smallest nonnegative classical solution of this initial / boundary value problem}.
 \qed
 \end{exmp}


\begin{exmp}{\bf Diffusion with quartic variance function.}
 \label{Tan} 
 Let us consider, for some given real number $\mub \in\R,$  the stochastic differential equation
 \begin{equation}
\label{tan}
\dx X(t) =\big( 1 + X^2 (t) \big) \big[ \,\dx W(t) +   ( \mub + X  (t) ) \,\dx t \, \big]\,, \quad X(0) = \xi   
\end{equation}
  of the form \eqref{1} with $\,\mathfrak{s} (x) = 1 + x^2\,$ and $\,\mathfrak{b} (x) = \mub +x\,$ for all $\,x \in I = \R\,.$  In this case the diffusion in natural scale  
$$
\dx X^o(t) =\left( 1 + (X^o (t))^2 \right)   \dx W^o(t)\,, \quad X^o (0) = \xi 
$$
of \eqref{13} does satisfy $\Prob^o(S^o=\infty)=1$ but does {\it not} satisfy the condition of \eqref{DS}: it is a strict local martingale, studied for example in \citet{CFR_qnv}. We have 
\begin{align*}
	 f(x) = \frac{\mub+x}{1 + x^2}\,,\qquad  F(x) = \mub \tan^{-1} (x) + \frac{\log(1+x^2)}{2}\,,\qquad   V(x) = \frac{1+\mub^2}{2}\,,
\end{align*}
 in the notation of Subsection~\ref{SS feynman}. Then  \eqref{17} and Corollary~1 and Lemma~2 in \citet{CFR_qnv}, which provide a distributional identity of the local martingale $X^o(\cdot)$ in terms of the Brownian motion $W^o (\cdot)$, applied with $$\widetilde{g}(x) = \cos(x+c)/\cos(c), \qquad \widetilde{f}(x) = \tan(x+c), \qquad \widetilde{h}(x) = (1+x^2)^{1/2} \exp(\mub \tan^{-1}(x))$$ and $C=1$, where $c = \tan^{-1} (\xi)$, yield the representation 
 \begin{align*}
\Prob (S>T) &=  \E^o \left[\left(\frac{1+(X^o (T))^2}{1+\xi^2}\right)^{1/2 } \exp \left(\mub \tan^{-1} \big(X^o (T)\big)  - \mub \tan^{-1} (\xi)  -  \frac{\,1+ \mub^2 \, }{2}\, T
\right)   \right]\\
	&=  \E^o \left[\exp \left(\mub W^o(T)   -  \frac{\mub^2 T}{2}    \right) \1_{\{\taub^o > T\}}
	   \right]
\end{align*}
with $\taub^o$ the first time that $W^o(\cdot)$ hits either
$$a =a (\xi) := -(\pi / 2) - \tan^{-1} (\xi) \qquad  \mathrm{or} \qquad b = b(\xi) := (\pi / 2) - \tan^{-1} (\xi).$$

Alternatively, we observe that the function $\mathfrak{b}(\cdot)$ is of the form   \eqref{30}, and thus,  following  Appendix~\ref{Doss}, 
 $$
 X(t) = \tan \left( W(t) + \mub  t + \tan^{-1} (\xi) \right), \qquad 0 \le t < S 
 $$
  identifies the explosion time as a first exit time for Brownian motion with drift
\begin{equation*}
 S = \inf \big\{ t \ge 0: W(t) + \mub t \notin (a, b) \big\}  
 \end{equation*}
from the open interval $(a,b)$. Formula~(3.0.2) on page~309 of \citet{Borodin_handbook} now computes the distribution of this exit time as
\begin{equation}
\label{29b}
 \Prob (S>T)=\int_T^\infty \exp\left( \frac{- \mub^2 t}{2}\right) \big(  \exp(\mub a) \Theta \left(t; b, b-a\right)+ \exp(\mub b) \Theta  \left(t; -a, b-a\right) \big) \,\dx t\,,
 \end{equation}
 where   
\begin{align*}
\Theta (t; u,v) :=  \sum_{k \in \mathbb{Z}}  {v-u + 2 k v   \over  \sqrt{ 2\pi t^3}} \exp \left( - { \left( v-u + 2 k v \right)^2 \over  2 t}  \right)
 \end{align*}
denotes a certain inverse Laplace transform.

In this example,  the explosion time $S$ has actually finite expectation. From Propositions~\ref{Prop3} and \ref{Prop1}, the function on the right-hand side of \eqref{29b} is the {\it smallest nonnegative classical solution of the Cauchy problem }
$$
{  \partial \mathcal{U}  \over \partial \tau} (\tau, x) ={   (1+x^2)^2 \over 2} {  \partial^2 \mathcal{U}  \over \partial x^2} (\tau, x)  + (\mub+x) (1+x^2) {  \partial \mathcal{U}  \over \partial x} (\tau, x), \quad (\tau, x) \in (0, \infty) \times \R
$$
with boundary condition $\,\mathcal{U} (0 , x) =1\,$ for all $ x  \in   \R$. \qed
\end{exmp}

\begin{exmp}{\bf Diffusion with cubic variance function.} 
 \label{3/2} 
 In a similar manner, we  consider the stochastic differential equation
$$
\dx X(t) =\big(   X (t) \big)^{3/2} \left [  \dx W(t) +   \left( \mub +  { 3 \over \,4\,} \big(X  (t)\big)^{1/2} \right) \dx t  \right], \qquad X(0) = \xi \in I  
$$
with state space $ I = (0, \infty)$, for some given real number $\mub \in \R$. This equation is  of the form \eqref{1} with $\mathfrak{s} (x) = x^{3/2}$, $\, \mathfrak{b} (x) = \mub  + (3/4) x^{1/2} $, and  the diffusion in natural scale  
$$
\dx X^o(t)  = \big(   X^o (t) \big)^{3/2}    \dx W^o(t), \quad X^o(0) = \xi 
$$
of \eqref{13} satisfies the conditions of \eqref{DS0}. We note that 
$$\mathfrak{f} (x) = \frac{\mub}{x^{3/2}} + \frac{3}{4x},\qquad F(x) =  \frac{\,-2\mub\,}{\,\sqrt{x\,}\,} +\log ( x^{3/4}),\qquad V(x) = \frac{\mub^2}{2} - \frac{3 x}{32}
$$
 in the notation of Subsection~\ref{SS feynman}, so \eqref{17} gives the representation
$$
\Prob_\xi (S>T)=  \exp\left(\frac{2 \mub}{\,\sqrt{\xi\,}\,}- \frac{\mub^2 T}{2}\right) \E^o \left[\left( {  X^o (T) \over \xi} \right)^{3/4}  \exp \left( { 3   \over 32} \int_0^T    X^o (T) \dx t  -\frac{2 \mub}{\,\sqrt{X^o(T)\,}\,}
\right)  \right] .
$$

This probability  can be computed explicitly via the observations in Appendix~\ref{Doss}. 
We obtain the diffusion $X(\cdot)$ explicitly as  
 $$
 X(t) = \left( { 1 \over \,\sqrt{\xi\,}\,} - {1 \over\, 2\,} \big( W(t) + \mub  t\big)    \right)^{-2}, \qquad 0 \le t < S.
 $$
 This shows  that the origin is inaccessible by the diffusion $X(\cdot)$, which can thus only explode to infinity, consistent with  Proposition~\ref{P conditions}(i). The explosion happens at the Brownian first passage time
 $$
 S = \inf \left\{ t \ge 0:  W(t) + \mub t = { 2 \over \, \sqrt{\xi\,}\,} \right\},
 $$
  whose distribution is of course well known, namely 
 $$
 \Prob_\xi (S>T)= 1-\int^T_0 \left( { 2 \over  \pi \xi  t^{3} } \right)^{1/2} \exp \left( - { 1 \over  2t} \left( { 2 \over \,\sqrt{\xi\,}\,} - \mub  t \right)^2 \right) \dx t \,=:\, U (T, \xi)\,;
 $$
see also Section~3.5.C in \citet{KS1}.
 In particular, with $\mub <0$ we have $\Prob (S < \infty) =   \exp(4 \mub / \sqrt{\xi})$; whereas, with $ \mub \ge 0$, we have $\Prob (S < \infty) =   1$. It is checked by direct computation, that the above function $U(\cdot, \cdot)$ solves the parabolic partial differential equation
$$
\frac{\partial \mathcal{U}}{ \partial \tau} (\tau, x) = { x^3 \over 2}  {  \partial^2 \mathcal{U}  \over \partial x^2} (\tau, x)  + \left( \mub  x^{3/2}+ { 3 \over \,4\,} \,x^2\right) { \partial \mathcal{U} \over \partial x} (\tau, x), \qquad  (\tau, x) \in (0, \infty) \times \R
$$
with  boundary condition $\mathcal{U} (0 , x) \equiv 1$ for all $x \in I$; indeed, from Proposition~\ref{Prop1},   $U (\cdot\,, \cdot)$ is the smallest nonnegative (super)solution of this linear parabolic equation.

We remark that a related example is discussed in Corollary~1 of \citet{Andreasen_2001}.
\qed
 \end{exmp}

   \begin{exmp}{\bf Explosion to infinity.} Let us now consider the situation with $ I = \R$, $\mathfrak{s} (\cdot) \equiv 1$ and $  \mathfrak{b} (x) = \exp( \beta x)$ for some $ \beta > 0$.   In this case the Brownian motion $ X^o (\cdot) = \xi + W^o(\cdot)$ has no explosions, i.e., $\Prob^o (S^o = \infty) =1$; the process $ X(\cdot)$ with dynamics
\begin{equation*}
 X(\cdot) \,=\, \xi + \int_0^{\, \cdot} \exp \big( \beta X(t) \big)\, \dx t + W(\cdot)
\end{equation*} 
      explodes (to $+ \infty$) in finite time, that is $ \Prob (S < \infty) =1$; and 
 $$
 \Prob (S>T) = \E^o \left[  \exp \left( { 1 \over  \beta} \left( e^{ \beta X^o (T)}- e^{ \beta \xi} \right)- { \beta \over 2} \int_0^T e^{ \beta X^o (t)} \dx t - { 1  \over 2} \int_0^T e^{2 \beta X^o (t)} \dx t \right)  \right] 
 $$
 holds for all $ T \in (0, \infty)$. From Formula~(1.30.7) on page~196 in \citet{Borodin_handbook}, we   obtain the distribution of the explosion time as 
\begin{align*}
	\Prob(S>T)  &= \beta^2   \int_{-\infty}^\infty  \int_0^\infty  { \exp \left( ( e^{ \beta z} - e^{ \beta \xi} ) / \beta - (\beta y/2) \right)\over 8 \sinh ( \beta y / 2)}
		\exp \left(  - {   e^{ \beta \xi} + e^{\beta z}   \over \beta   \tanh (\beta y /2) }\right) \\ 
		&~~~~~~~~~~~~~~~~~~~~~~~~~~~~~~~~~~~~~~~~~~~~\cdot  \mathbf{ i}_{\beta^2 T / 8}   \left(  {  2 e^{ \beta (\xi  +  z )  /2}  \over \beta   \sinh (\beta y /2) }\right) \dx y \, \dx z\,,
		  \end{align*}
where 
\begin{equation*}
\mathbf{ i}_{t} (z) := {\cal L}^{-1}_t \left( I_{\sqrt{t}} \,(z) \right)=
{   z  \exp(\pi^2 / (4 t)) \over \pi \sqrt{ \pi t}} \int_0^\infty \exp\left(-z  \cosh (u) - \frac{u^2}{4 t }\right)  \sinh (u)  \sin \left( {  \pi  u \over 2  t} \right)  \dx u
\end{equation*}
denotes an inverse Laplace transform related to the modified {Bessel} function of the second type in \eqref{ModBes}. We obtain
\begin{align*}
	\Prob(S>T)  
		&= \frac{2 \exp(2 \pi^2/(\beta^2 T)) \sqrt{2} \eta}{\pi \beta \sqrt{\pi T}} \int_{0}^\infty  \int_0^\infty \int_0^\infty   { \exp \left( \zeta^2 - \eta^2  - \psi  \right)\over   (\sinh (\psi))^2}
		\exp \left(  - {  \eta^2 +\zeta^2   \over   \tanh (\psi) }\right) \\
		&~~~~~~~~~~~~~~~~~~~~~~~~~~~~~~~~~~~\cdot   \exp\left(-\, 2  \eta  \zeta \, {  \cosh (u) \over \,  \sinh (\psi)\, }  - \frac{2 u^2}{\beta^2 T }\right)  \sinh (u)  \sin \left( {  4 \pi  u \over \beta^2 T} \right) 
 \dx u \, \dx \psi \, \dx   \zeta,
  \end{align*}
 after applying the substitutions  of $\,\beta y/2\,$ by $\psi$, and of $\exp(\beta z/2)  /\sqrt{\beta\,}\,$ by $\zeta$, and the change of variable $ \eta =  \exp\left(\beta \xi/2\right)/\sqrt{\beta\,}$.
\qed
     \end{exmp}

\appendix

\section{Appendix:  Feller  test }
\label{FT}

A well-known criterion for deciding whether the diffusion $X(\cdot)$ can explode or not (to wit, whether we have $\Prob(S<\infty)>0$  or $ \Prob(S<\infty)=0$) is Feller's test; see Theorem~5.5.29 in \citet{KS1}. This criterion relies on the {\it Feller  test  function} 
defined as
\begin{align}  
\label{E v}
w(x)  := \int_c^x \left(  \frac{\exp \left(2 F(z) \right)}{\mathfrak{s}^2 (z)}  \left( \int_z^x \exp \big( - 2 F(y) \big)\right)   \dx y  \right) \dx z \,, \qquad x  \in I 
\end{align}
for   some fixed constant $c \in I$,  where $\, F(\cdot)\,$ is the function of \eqref{F_anti}. 
   Feller's test then states that $\Prob(S=\infty) = 1$ holds, if and only if $w(\ell+) = w(r-) = \infty\,$.  Moreover, the finiteness or nonfiniteness of $w(\cdot)$ does not depend on the choice of the constant $c$.  Alternative characterizations are discussed  in Subsection~\ref{SS:equivalent}.

 Since the function $1/\mathfrak{s}^2(\cdot)$ is locally square-integrable thanks to the assumption in \eqref{2}, and since the anti-derivative $F(\cdot)$ is continuous, we have $\,w(\ell_n) + w(r_n) < \infty\,$  and, in particular,    
\begin{align}  \label{eq:Sn finite}
	\Prob(S_n < \infty) = 1\,,
\end{align}
 for all $n \in \N$; in fact, we even have $\E[S_n] < \infty$, see Proposition~5.5.32 in \citet{KS1}.

There are several situations in which the Feller test function  $w(\cdot)$ can be simplified:
\begin{itemize}
	\item If the function $\mathfrak{s}(\cdot)$ is differentiable and $\mathfrak{b}(\cdot) = a \,\mathfrak{s}^\prime(\cdot)$  for some $\,a \in \R\,$, then
		\begin{align*}
			w(x) =  \int_c^x \left(  \mathfrak{s}^{2a-2} (z) \int_z^x \mathfrak{s}^{-2a}(y) \dx y \right) \dx z.
		\end{align*}
		In particular, we have the following two special cases:
		\begin{itemize}
			\item $a = 1/2$: 
		\begin{align} \label{eq: feller 1/2}
			w(x) =  \int_c^x \left(  \frac{1}{\mathfrak{s} (y)} \int_c^y\frac{1}{\mathfrak{s}(z)} \dx z \right) \dx y = \frac{1}{2} \left(   \int_c^x \frac{1}{\mathfrak{s}(z)} \dx z\right)^2;
		\end{align}
			\item $a = 1$: 
		\begin{align*}
			w(x) =  \int_c^x \frac{z-c}{ \mathfrak{s}^{2} (z)} \dx z.
		\end{align*}
		\end{itemize}
	\item If $\mathfrak{b}(\cdot) = a \, \mathfrak{s}(\cdot)$ for some $a \in \R \setminus \{0\}$, then
		\begin{align*}
			w(x) = \frac{1}{2a} \int_c^x \frac{\,1-\exp(2a(z-x))}{\mathfrak{s}^2(z)\,}\, \dx z\,.
		\end{align*}
	\item If $\mathfrak{b}(x) = a \, \mathfrak{s}(x)/x$ for all $x \in \R$ for some $a \in \R \setminus \{1/2\}$, then
		\begin{align} \label{F generalized}
			w(x) = \frac{1}{1-2a}  \int_c^x \frac{ \,x(z/x)^{2a}-z\,}{\mathfrak{s}^2(z)} \, \dx z\,.
		\end{align}
		In particular, with $a=0$, we have
		\begin{align}  \label{F}
			w(x) =  \int_c^x \frac{\,x-z\,}{ \mathfrak{s}^{2} (z)} \,\dx z\,.
		\end{align}
\end{itemize}
 From these observations we obtain the following useful corollary.
 
\begin{prop}{\bf Conditions for explosions in special cases.} 
\label{P conditions}
	Feller's test simplifies in the following cases:
	\begin{itemize}	
		\item[(i)]  Suppose that the function $\mathfrak{s}(\cdot)$ is differentiable (without loss of generality, we then assume that $\mathfrak{s}(\cdot) >0$ on $I$), and that $\,\mathfrak{b}(\cdot) = \mathfrak{s}^\prime(\cdot)/2\,$. Then $\,\Prob(S=\infty) = 1\,$ holds, if and only if for some $c \in I$ we have 
		\begin{align*}
			\int_\ell^c \frac{\,\dx z\,}{\mathfrak{s}(z)}  = \infty = \int_c^r \frac{\,\dx z\,}{\mathfrak{s}(z)}   \,.
		\end{align*}
				\item[(ii)] Suppose that  $I = (0,\infty)$ and $\mathfrak{b}(\cdot) = 0$. Then $w(0+) = \infty$ holds, if and only if we have 
\begin{equation}
\label{DS0}
\int_{0}^1  \frac{z}{\mathfrak{s}^2 (z) } \dx z =\infty\,;
\end{equation}			
		 moreover, $w(\infty ) = \infty$. Therefore,   $\Prob(S=\infty) = 1$ holds in this case,   if and only if \eqref{DS0}  does.
		\item[(iii)] Suppose that  $\,I = (0,\infty)\,$ and $\,\mathfrak{b}(x) = \mathfrak{s}(x)/x\,$ for all $\,x \in I$. Then $w(\infty )= \infty$ holds, if and only if we have 
\begin{equation}
\label{DS}
\int_{1}^\infty  \frac{z}{\mathfrak{s}^2 (z) } \dx z =\infty\,;
\end{equation}			
		moreover, $w(0+) = \infty\,$. Therefore,   $\Prob(S=\infty) = 1$ holds in this case,   if and only if \eqref{DS} does.
	\end{itemize}
\end{prop}
\begin{proof}
	Part~(i) is an application of the identity in \eqref{eq: feller 1/2} and Feller's test. For part~(ii), we first observe that
the condition $w(0+) = \infty\,$, along with the representation in \eqref{F} with $c=1$, imply that \eqref{DS0} holds. For the reverse direction, we assume that \eqref{DS0} holds and define $\,m(y) := \int_y^1 \mathfrak{s}^{-2}(z) \dx z\,$ for all $y \in [0,1]$.  If $\,\limsup_{y \downarrow 0} (y m(y)) < \infty\,$ holds, then $w(0+) = \infty$ by \eqref{F} with $c = 1$. If $\limsup_{y \downarrow 0} (y m(y))) = \infty$ holds, we observe that we may rewrite \eqref{F} as 
$$w(0+) =  \lim_{x \downarrow 0} \int_x^1 \frac{z-x}{\mathfrak{s}^2(z)} \dx z = \lim_{x \downarrow 0} \left(\int_x^1 \int_x^z \frac{1}{\mathfrak{s}^2(z)} \dx y\right) \dx z = 
\int_0^1 m(y)  \dx y \geq \int_0^\mathfrak{y} m(y) \dx y \geq \mathfrak{y} m(\mathfrak{y})$$ for all $\mathfrak{y} \in [0,1]$, which yields $w(0+) = \infty$. Moreover, we have $w(x) \geq (x-2) \int_1^2 \mathfrak{s}^{-2}(z) \dx z$ for all $x \geq 2$, which yields   $w(\infty ) = \infty$.

For part~(iii), we use \eqref{F generalized} with $a = 1$ and $c = 1$ to obtain the representation
		\begin{align} \label{F generalized2}
			w(x) =   \int_1^x \frac{z - z^2/x}{\mathfrak{s}^2(z)} \dx z =  \int_1^x \left(\frac{1}{y^2} \int_1^y \frac{z^2}{ \mathfrak{s}^2(z)} \dx z \right) \dx y.
		\end{align}
	In the same manner as above we have $w(0+) = \infty\,$ under \eqref{DS}, and   need only   show that \eqref{DS} implies $w(\infty ) = \infty\,$. We  may assume again  $\,\limsup_{y \uparrow \infty}  (k(y)/y) = \infty$, where $\,k(y) = \int_1^y z^2 \mathfrak{s}^{-2}(z) \dx z\,$ for all $y \in [1,\infty)$, as otherwise the statement is clear. Under this assumption, we obtain from \eqref{F generalized2} that
	\begin{align*}
		w(\infty ) \,\geq \,\int_y^\infty \frac{\,k(z)\,}{z^2} \, \dx z \, \geq \, k(y)  \int_y^\infty \frac{\dx z}{z^2} \, = \,\frac{k(y)}{y}
	\end{align*}
holds	for all $y \geq 1$, which concludes the proof.
\end{proof}

Of course, additional statements in the form of the last proposition can be proved; we focused here on those   needed in the body of the paper. 
An alternative proof of the equivalence in part~(ii) under slightly stronger assumptions and using the Ray-Knight theorem,  appears in Theorem~1.4 of \citet{DShir}.  In Example~\ref{Ex h-transform}, we discuss the setup of part~(iii) in Proposition~\ref{P conditions},  and its connection to part~(ii) will then become clearer; see also Corollary~\ref{C:mg property}.

\section{Appendix: Explosions as Brownian exits via    Lamperti transformations}
 \label{Doss}

In   Subsection~\ref{GenGir} we saw how to remove drifts by changing the underlying probability measure. We discuss now  ways to transform the dispersion term into a constant,  by distorting the space as  $Y(\cdot) = h(X(\cdot))$, for some strictly increasing and continuous function $h: (\ell,r) \rightarrow (\widetilde{\ell}, \widetilde{r}\,)$ and suitable $-\infty \leq \widetilde{\ell} < \widetilde{r} \leq \infty$. We shall assume in this section that the function $ \mathfrak{s} (\cdot)$ is   continuously differentiable on the interval $I=(\ell,r)$; without loss of generality, we shall also assume that $\mathfrak{s}(\cdot)$ is strictly positive. 
\newpage
We shall consider the function 
\begin{align*}
	h_c(x) = \int_c^x \frac{\,\dx z\,}{\mathfrak{s}(z)} \,, \quad  x \in (\ell,r)
\end{align*}
for   some $\,c \in (\ell, r)$. We observe that $h_c(\cdot)$ is strictly increasing and twice differentiable. We set $\,\widetilde{\ell} = h_c(\ell) := \lim_{x \downarrow \ell} h_c(x) \in [-\infty, \infty)$ and $\,\widetilde{r} = h_c(r) := \lim_{x \uparrow r} h_c(x) \in (-\infty, \infty]$ and define the process $Y(\cdot)$ via 
$$
Y(t) := h_c(X(t)) ~ \text{ for all } t \in [0, S) ~~\text{ and } ~~Y(t) = \lim_{u \uparrow S} h_c(X(u)) ~ \text{ for all } t \in [S, \infty).
$$
 It is clear that   $\lim_{t \uparrow S} Y(t) \in \{\widetilde{\ell}, \widetilde{r}\,\}$ holds on $\{S < \infty\}$, and that the new process $Y(\cdot)$ leaves its state space $\widetilde{I} :=(\widetilde{\ell}, \widetilde{r}\,)$ at exactly the time $S$.  In particular,  the (distribution of the) explosion time $S$ of $X(\cdot)$ is exactly the (distribution of the) explosion time $\widetilde{S}$ of $Y(\cdot)$.

With $\,\vartheta_c: \widetilde{I} \rightarrow I$ denoting the inverse function of $h_c$, simple stochastic calculus yields that
\begin{align*}
	\dx Y(t) = \left(\mathfrak{b}(\vartheta_c(Y(t))) - \frac{\mathfrak{s}^\prime(\vartheta_c(Y(t)))}{2 }\right) \dx t + \dx W(t), \qquad Y(0) = h_c(\xi) =: \widetilde{\xi}  \in \widetilde{I}
\end{align*}
hold for all $t \in [0, S)$.  In particular, with the function $\,\nu: \widetilde{I} \rightarrow I\,$ defined by 
\begin{align*}
	\nu(y) := \mathfrak{b}(\vartheta_c(y)) -\frac{1}{\,2\,} \,\mathfrak{s}^\prime(\vartheta_c(y))\,,\qquad y \in \widetilde{I}
\end{align*}
we have the simple dynamics
\begin{align*}
	\dx Y(t) = \nu(Y(t)) \dx t + \dx W(t), \qquad Y(0) = h_c(\xi) = \widetilde{\xi}
\end{align*}
for all $t \in [0, S)$. 
As \citet{Lamperti1964} stresses \citep[see also Section~3.4 in][]{McKean_1969}, this equation can be solved pathwise by simple Picard iterations, without any need for stochastic integration  or other probabilistic tools, as long as the function $\nu(\cdot)$ is Lipschitz continuous.  In particular, if 
\begin{align} \label{30}
	\mathfrak{b}(\cdot) = \frac{1}{\,2\,} \,\mathfrak s^\prime(\cdot) + \mub
\end{align}
for some constant $\mub \in \R$, the computation of the time to explosion reduces to to computing the distribution of the time to explosion for a Brownian motion with drift.

 This approach has been generalized in an effort to study the pathwise solvability of stochastic differential equations
 by  \citet{Lamperti1964}, \citet{Doss1977}, \citet{Sussmann1978}, and \citet{Karatzas_Ruf_DS}.

\section{Appendix: A technical  result on uniqueness in distribution}
 \label{A:uniqueness}

 In this appendix, we revisit the diffusion $X^o(\cdot)$ of Subsection~\ref{sec1.2}. Theorem~5.5.7 in \citet{KS1} yields the uniqueness in the sense of the probability distribution of the stochastic integral equation in \eqref{13}; however, the proof of Theorem~\ref{Thm1} requires a slightly stronger uniqueness statement. Towards this end, and using the notation of the paragraph right before Theorem~\ref{Thm1}, we call a function   $\thetab: C ([0, \infty)) \rightarrow [0,\infty]$ a  {\it stop-rule,} if 
 \begin{equation*}
\big\{ \mathfrak{w} \in C ([0, \infty)) : \thetab (\mathfrak{w}) \le t \big\}  \in  {\cal B}_t:= \varphi^{-1}_t  ( {\cal B}  )  \qquad \text{holds for all } ~\,0 \le t < \infty\,.
\end{equation*}

\begin{prop}{\bf Uniqueness up to stopping times.}  \label{P:uniqueness appendix}
Let  $\thetab$ denote a stop-rule satisfying $\mathfrak{w} (  t)\in I$  for all  $\, \mathfrak{w} \in  C ([0, \infty))\,$ and $\,0 \le t < \thetab (\mathfrak{w})\,$.  The   solution of the
``stopped'' version of the stochastic integral equation in \eqref{13}, namely
\begin{equation}
\label{113}
\widehat{X} (\cdot)  = \xi + \int_0^{\, \cdot \wedge \thetab (\widehat{X})}  \mathfrak{s} \big( \widehat{X} (t) \big) \, \dx \widehat{W} (t)\,,
\end{equation}
is unique in the sense of the probability distribution.
\end{prop}
\begin{proof}   
Let us consider {\it any} weak solution of \eqref{113} and denote $\varrho = \thetab (\widehat{X})$.
The solvability of \eqref{113} implies that  the time change $\,A  (\cdot) :=\int_0^{ \cdot}  \mathfrak{s}^2 (\widehat{X} (t))\1_{\{\varrho >t\}} \dx t \, $ is well-defined, and we note that this process is the quadratic variation of the continuous local martingale  $M  (\cdot)= \int_0^{ \cdot  }  \mathfrak{s}  (\widehat{X} (t)) \1_{\{\varrho >t\}}\dx \widehat{W}(t) $.

\smallskip
According to the {Dambis-Dubins-Schwarz} theory \citep[see Theorem~3.4.6 and Problem~3.4.7 in][]{KS1} there exists a standard Brownian motion $B (\cdot)$ on (an extension of) the underlying probability space, such that 
$$\widehat{X}  (\cdot) =\xi + \int_0^{\cdot}  \mathfrak{s} \big( \widehat{X}  (t) \big) \1_{\{\varrho >t\}} \dx \widehat{W} (t) = \xi + B  ( A  (\cdot) ).
$$
We consider now the inverse time change $ \Gamma  (\theta) := \inf \{ t \ge 0 : A  (t) > \theta\}$ for all $0 \le \theta < A  (\varrho)$ and $ \Gamma  (\theta) :=\infty$ for all $ \theta \ge A  (\varrho)$, and note that
$$
\Gamma^\prime (\theta)= { 1 \over  A^\prime (\Gamma  (\theta))}={ 1 \over   \mathfrak{s}^2 \big( \widehat{X}  (\Gamma  (\theta))\big)}\,, \qquad \hbox{thus} \qquad \Gamma   (\theta)=\int_0^\theta { \dx r \over \mathfrak{s}^2 ( \xi + B  (r))}  
$$
for all $ 0 \le \theta < A  (\varrho)$. Next, we define the function $ \mathfrak{s}_* (x) :=  \mathfrak{s} (x) \1_{  (\ell, r) } (x)+ \1_{ \R \setminus (\ell, r)}(x) $ and the corresponding time change
$$
\Gamma_*   (\theta)=\int_0^\theta { \dx r \over  \mathfrak{s}_*^2 ( \xi + B  (r))} \, , \qquad 0 \le \theta <\infty\,,
$$
along with its inverse $A_* (t):=  \inf \{ \theta \ge 0 : \Gamma_*  (\theta) > t\}$  for all $0 \le t <\infty$. 

We note that we have the ordinary integral equations 
\begin{align}
	A(\cdot \wedge \varrho) &= \int_0^{\,\cdot \wedge \varrho} \mathfrak{s}^2\big(\xi + B(A(t))\big) \, \dx t\,,\label{eq:C2}\\
	A_*(\cdot \wedge \Gamma_*(\taub)) &= \int_0^{\,\cdot \wedge \Gamma_*(\taub)} \mathfrak{s}^2\big(\xi + B(A_*(t))\big) \, \dx t \nonumber
\end{align}
in the notation of \eqref{TAU}. We also have $\varrho \leq \Gamma_*(\taub)$, a consequence of our    assumption that $\xi + B(A(t)) \in I$ for all $t < \varrho$. The uniqueness of solutions to \eqref{eq:C2} implies then    $A (\cdot  \wedge \varrho) =  A_*(\cdot \wedge \varrho)$;  therefore, the process $X_*  (\cdot) := \xi + B  ( A_*  (\cdot) ) $ satisfies  $ \widehat{X}(\cdot )= \widehat{X} (\cdot \wedge \varrho) =  X_*(\cdot \wedge \varrho)$ and 
$$
\varrho= \thetab \big(\widehat{X} (\cdot)\big) =  \thetab \big(\widehat{X} (\cdot \wedge \varrho)\big)  =  \thetab \left(X_* (\cdot \wedge \varrho)\right) = \thetab \left( X_* (\cdot)\right) .
$$
We note that the process $A_*(\cdot)$ is $\,\F^{B } (\infty)-$measurable, thus so is $X_*(\cdot)$ and hence also $\widehat{X}(\cdot)$.  In particular, the distribution of $ \widehat{X} ( \cdot) $ is determined uniquely. 
\end{proof}

\bibliography{aa_bib}{}
\bibliographystyle{apalike}

\end{document}